\documentclass[a4paper,10pt]{amsart}
\usepackage{amsfonts,amssymb,amsmath,amsthm}
\usepackage{url}
\usepackage{enumerate}
\usepackage{color}
\usepackage[color,all]{xy}
\usepackage[all]{xy}

\usepackage{amssymb, amsmath}
\usepackage{amscd}
\numberwithin{equation}{section}
\usepackage{epsfig}
\usepackage{amsmath}
\usepackage{amsfonts,amssymb,amsopn}

\usepackage{amsthm}
\usepackage{hyperref}
\usepackage{verbatim}
\usepackage{color}
\usepackage[color,all]{xy}
\usepackage{graphicx}

\newcommand{\Ker}{\mathrm{Ker}}

\newcommand{\im}{\mathrm{Im}\,}

\newcommand{\Pic}{\mathrm{Pic}}

\newcommand{\Ext}{\mathrm{Ext}}
\newcommand{\rk}{\mathrm{rk}}
\newcommand{\codim}{\mathrm{codim}}

\newcommand{\comma}{\raisebox{0.5ex}{,}}

\newcommand{\Cliff}{\operatorname{Cliff}}
\newcommand{\red}{\operatorname{red}}
\newcommand{\Coker}{\operatorname{Coker}}

\newtheorem{theorem}{{\textbf Theorem}}[section]
\newtheorem{proposition}[theorem]{{\textbf Proposition}}
\newtheorem{corollary}[theorem]{{\textbf Corollary}}
\newtheorem{lemma}[theorem]{{\textbf Lemma}}

\newtheorem{defit}[theorem]{{\textbf Definition}}
\newtheorem{Properties}[theorem]{{\textbf Properties}}
\newtheorem{conjecture}[theorem]{{\textbf Conjecture}}
\newtheorem{remit}[theorem]{{\textbf Remark}}
\newtheorem{convention}{{\textbf Convention}}

\newenvironment{remark}{\begin{remit}\rm}{\end{remit}}
\newenvironment{definition}{\begin{defit}\rm}{\end{defit}}

\title[Linear Stability and rank two Clifford Indices 
 ]{  Linear Stability and rank two Clifford Indices of Algebraic Curves with applications }

\author{Ali Bajravani}
\address{A.\ Bajravani: Institut für Mathematik, Humboldt Universität zu Berlin, Germany\\
	Department of Mathematics, Azarbaijan Shahid Madani University, Tabriz, I.\ R.\ Iran, P.\ O.\ box 53751-71379}
\email{ali.bajravani@hu-berlin.de; bajravani@azaruniv.ac.ir}

\address{A.\ Ortega: Institut für Mathematik, Humboldt Universität zu Berlin,\\
	Unter den Linden 6, D-10099 Berlin, Germany. 
}
\email{ortega@math.hu-berlin.de}

\author{Angela\ Ortega}


\begin{document}

\begin{abstract}
We prove that any vector bundle computing the rank-two Clifford index of a smooth projective algebraic curve is linearly semistable. We also identify conditions under which such bundles become linearly stable, thereby addressing a question posed by A. Castorena, G. H. Hitching, and E. Luna in the rank-two case.
Furthermore, we demonstrate that, in certain special cases, this property is equivalent to the (semi)stability of the associated Lazarsfeld–Mukai bundles. This yields a positive solution, in specific cases, to a generalized version of a conjecture proposed by Mistretta and Stoppino. 
We also study the moduli space $S_0(n, d, 5)$ of 
generated $\alpha$-stable coherent systems of type (n,d,5) for small values of $\alpha$ and n=2,3. We show that a general element of an irreducible component $X \subseteq S_0(2, d, 5)$ or $X \subseteq S_0(3, d, 5)$ is linearly stable 
whenever $2\delta_2\leq d\leq \frac{3g}{2}$.  As 
an application of this,  we prove 
that Butler's conjecture holds non-trivially for $S_0(2,d,5)$ within the given range for $d$.
\end{abstract}

\maketitle

\vspace{-0.4cm}

\section{Introduction}

Let $X$ be an irreducible nondegenerate projective variety and let $L\rightarrow X$ be a globally generated line bundle of degree
$d$. We denote by $\psi_L: X \rightarrow \mathbb{P}:=\mathbb{P}(H^0(L)^*)$, the morphism induced by $L$.  The reduced degree of $X$, is defined as
$$\red \deg(X):={\deg(\psi_L(X))\over{\codim_{\mathbb{P}}X+1}}\cdot$$
Equivalently, this can be expressed as 
$$\red \deg(X)=\frac{\deg L}{h^0(L)-\dim X}\comma    $$      
The invariant $\red \deg(X)$ plays a significant role in the study of the geometry of $X$; it satisfies the inequality
 $$\red \deg(X)\geq 1, $$
 and a classical result of Eisenbud and Harris asserts the following:  if $X\subset \mathbb{P}^n$ is smooth and is not contained in any hyperplane $\mathbb{P}^{n-1}\subset \mathbb{P}^n$, then $\red \deg(X)$ attains its minimum value if and only if $X$ is either a rational normal scroll or the Veronese surface in $\mathbb{P}^5$ (see \cite{EH}).
 
Mumford introduced in \cite{Mum} the notion of linear stability for projective varieties, formulated as a property of the linear system $L$ embedding a variety $X \subset \mathbb{P}(H^0(L)^*)$. Accordingly, a variety $X$ of dimension  $r$ is called linearly stable (respectively, linearly semistable) if, for all subspaces $W\subset H^0(L)$ such that the image of
 the projection  
 $$\pi_W: \mathbb{P}(H^0(L)^*) \rightarrow \mathbb{P}(W^*)$$
 induced by $W$ has dimension $r$,  the following inequality holds: 
$$\red \deg(\pi_W(X)) > \red \deg (X) \quad  (\mbox{respectively}, \quad \!\!\!\!\!\geq). $$

Recently in \cite{CHE} has been proposed a definition extending the notion of linear (semi)stability to higher rank for generated coherent systems.

Let $E$ be a vector bundle over a smooth projective curve $C$, generated by a subspace of sections $V \subseteq H^0(C,L)$; that is the evaluation map 
$V \rightarrow E_{|x}$ is surjective for every point $x\in C$. Note that $\dim V > \rk E$, whenever $E$ is non-trivial. In this context, the linear slope of $E$ with respect to $V$ is defined as
$$\lambda(E, V):=\frac{\deg E}{\dim V-\rk E}\cdot
$$
When $V=H^0(L)$ we write $\lambda(E)$ for brevity.

A globally generated coherent system $(E, V)$ of type
$(n,d,k)$; that is, a pair consisting of a vector bundle $E$ of rank $n$ and degree $d$, generated by a subspace $V\subseteq H^0(E)$ with $\dim V=k$; is said to be linearly (semi)stable if, for every generated subsystem $(F, W)$, the inequality
$$\lambda(E, V) < \lambda(F, W) \quad  (\mbox{respectively}, \quad \!\!\!\!\!\leq)
$$
holds.

It follows directly from the Riemann–Roch and Clifford Theorems that the canonical bundle $K_C$ of a smooth curve $C$ is linearly semistable. In \cite{BS}, M. A. Barja and L. Stoppino showed that if $V<H^0(K_C)$ is a general subspace of codimension $c\leq \frac{\Cliff(C)}{2}$, then $(K_C, V)$ is linearly stable. 

Recall that among line bundles $L$ of fixed degree $d$ satisfying $h^0(L), h^1(L)\geq 2$, those computing the $\Cliff(C)$ attain the maximum number of sections. As a result, the linear slope 
$\lambda$ associated with such line bundles, is expected to be minimal, suggesting that these line bundles should be linearly (semi)stable.  
This expectation was confirmed in \cite[Proposition 3.3]{MS}, where the authors established the linear semistability of such line bundles. 
The proof is based on a comparison between the Clifford index of $L$ and the Clifford indices of its generated subbundles.

Since slope semistability is fundamental in the definition of higher rank Clifford indices (originally introduced by Lange and Newstead) the argument used in the case of line bundles does not directly generalize to higher rank. Nevertheless, we prove that vector bundles computing the rank-two Clifford index of the curve $C$ are linearly semistable.

We begin by observing that for any line subbundle $L$ of a generated semistable and rank $n$ vector bundle $E$, generated by a subspace $V\subseteq H^0(E)$ and satisfying a certain numerical condition, one can bound the number of sections of $V$ in $L$ in terms of $\dim V$ and the rank of  $E$. 
As a consequence, we show that the linear slope of any globally generated invertible subsheaf of $E$ is greater than $\lambda(E)$.
Then we prove that the Clifford index of a globally generated (not necessarily semistable) rank-two vector bundle can be effectively compared with $\Cliff(C)$ under certain additional assumptions (see Lemma \ref{LNlem}). These observations allow us to establish the following theorem, which addresses Question $5$ in \cite{CHE} in the rank-two case.

\begin{theorem}\label{introthm} (\textit{Theorem \ref{indicecomputingthm1}, Theorem \ref{indicecomputingthm0}})
Let $C$ be a non-hyperelliptic curve and
let $E$ be a vector bundle computing $\Cliff_2(C)$, satisfying $\mu(E)\leq g-1$. Then $E$ is linearly semistable. Moreover, $E$ fails to be linearly stable if and only if one of the following conditions holds:\\
(i) $E$ contains a globally generated line subbundle $L$ with $h^0(L)=\frac{h^0(E)}{2}$, or \\
(ii) $E$ contains a rank-two globally generated locally free subsheaf $T$ such that $\Cliff(T)=\Cliff_2(C)$.
\end{theorem}
Our approach to proving Theorem \ref{introthm} differs from the method proposed in \cite{CHE}. Since Lemma \ref{LNlem} appears unlikely to generalize to higher ranks, determining whether vector bundles computing
$\Cliff_n(C)$ are linearly semistable for
$n\geq 3$ remains an open challenge.\\

On the other hand, the slope (semi)stability of the Lazarsfeld–Mukai bundle associated with a generated coherent system $(E, V)$, namely the kernel bundle of the surjective evaluation morphism 
$$\phi_{E, V}: V\otimes \mathcal{O}_C \rightarrow E$$
 denoted by $M_{E, V}$, implies the linear (semi)stability of $(E, V)$;
but the converse does not hold in general. Various counterexamples are known, see for instance \cite[Theorem 1.1]{CMT}. However, there are cases in which linear (semi)stability does imply the (semi)stability of the corresponding Lazarsfeld–Mukai bundle. In this direction, Castorena and Torres López showed in \cite{CT} that linear (semi)stability implies (semi)stability for any generated line bundle on a general curve.
  
For arbitrary curves, Mistretta--Stoppino formulated a pioneering conjecture asserting that
if $(L, V)$ is a $g^r_d$, that is a coherent system of type $(1, d, r+1)$, satisfying
$$d-2r\leq \Cliff(C),$$
then the linear (semi)stability of $(L, V)$ implies the (semi)stability of the associated Lazarsfeld--Mukai bundle $M_{L, V}$ (\cite[Conjecture 6.1]{MS},  Conjecture \ref{MSconjecture}). They verified the conjecture in various cases, in particular the case $V=H^0(L)$, thereby concluding the (semi)stability of $M_L$ for line bundles computing $\Cliff(C)$. 

We propose a higher rank version of  
Mistretta--Stoppino's Conjecture
(see \ref{MSBconjecture} (ii))
 and verify it in certain cases for rank $n=2$. The first key idea in our proof is that 
 the so-called Butler diagram associated to $E$ can be related to the Butler diagrams of $L\subset E$ with $h^0(L)=2$, and that one  of its corresponding quotient bundle  (Proposition \ref{diagramlem}). The second crucial point is that, whenever $h^0(L)\geq 2$, one can compare the linear slopes of $L$, $E$ and $E/L$. Using these two 
observations, we prove
\begin{theorem}\label{introthm2} (\textit{Theorem \ref{MSconjecturerank=2}, Theorem 
\ref{h0_less_than 6}})
Let $E$ be a vector bundle computing $\Cliff_2(C)$ and suppose that either $E$ admits a line subbundle $L$ with $h^0(L)=2$ or $h^0(L)\leq 6$. Then $E$ is linearly (semi)stable if and only if the associated Lazarsfeld--Mukai bundle $M_E$ is (semi)stable.
\end{theorem}
\vspace{0.2cm}

The concept of linear stability for generated coherent systems is closely related to the Butler conjecture (see \cite{BT},  \cite{CH2024}
 and \cite{MS}). This conjecture predicts that, on a general curve $C$, the Lazarsfeld–Mukai bundle $M_{E, V}$ is semistable 
whenever $(E, V)$ is general $\alpha$-stable coherent system, for small values of  $\alpha$ (Conjecture \ref{Butconjecture}). 
Butler's conjecture is known to hold for rank-one general generated coherent systems on general curves (\cite{BBN} and \cite{FL}). In higher rank,  the conjecture holds for coherent systems $(E, V)$ of type $(2, d, 4)$ in certain range of $d$ (\cite{Ortetal}). In \cite{CH2024}, the authors proved the conjecture for coherent systems of type $(2,d,5)$ under the assumptions that either $d=2\delta_2$, or $d=2\delta_2-1$ and $g \equiv 3 \textrm{ mod } 2$ (see Definition 2.2)
for curves of genus $g\geq 18$. They proved this result by relating the stability of $M_{E, V}$ not only to the 
linear stability of a generated coherent system $(E,V)$
but also to its $\alpha_L$-stability  for sufficiently large $\alpha_L$. 
In contrast, inspired by the ideas in \cite{Ortetal} and relying on Lemma \ref{LNlem}, we take a different approach. Specifically, we analyze the elements of $S_0(2, d, 5)$ and $S_0(3, d, 5)$ from the linear stability point of view and show that either a coherent system $(E, V)$ is linearly stable which, in our setting, is equivalent to the stability of its Lazarsfeld–Mukai bundle, or the locus of systems where this property does not hold, is less than the expected dimension of 
$S_0(2, d, 5)$.  We obtain the following theorem: 
\begin{theorem}\label{Butler}(\textit{Theorem \ref{Butlerfor2d5}})
Suppose $C$ is a general curve and $2\delta_2\leq d\leq \frac{3g}{2}$. Then, the Butler conjecture holds non-trivially for coherent systems of type $(2, d, 5)$.
\end{theorem}

\noindent \textit{Notations:} Throughout the paper, $C$ will denote a complex projective smooth curve. The canonical line bundle over $C$ will be denoted by $K_C$. For any sheaf $E$ over $C$, we abbreviate $H^i(C, E)$ and $h^0(C, E)$ to $H^i(C)$ and
$h^i(C)$, respectively.

\section{Preliminaries}
\subsubsection{Clifford Indices and the Gonality Sequence}
We begin by recalling the classical Clifford index and the higher Clifford indices of $C$, as introduced by Lange and Newstead in \cite{LNe22}.
\begin{definition}\label{Cliffordindices}
(i) The Clifford index of $C$ is defined to be 
\begin{align}\label{Cliffordindex}
\Cliff(C):=\min\{\Cliff(L): \mbox{$L$ is a line bundle with } h^0(L)\geq 2, \ h^1(L)\geq 2\},
\end{align}
in which $\Cliff(L):=d-2h^0(L)+2$.

\noindent (ii) If $E\rightarrow C$ is a vector bundle of rank $n$ and degree $d$, then 
$$\Cliff_n(C):=\small{\inf\{\Cliff(E): E \!\!\!\!\quad \mbox{is semistable with}\!\!\!\!\quad h^0(E)\geq 2n, \ \mu(E)\leq g-1\}},$$
$$\gamma_n(C):=\small{\inf\{\Cliff(E): E \!\!\!\!\quad \mbox{is semistable with}\!\!\!\!\quad h^0(E)\geq n+1, \ \mu(E)\leq g-1\}},$$
where $\Cliff(E):=\mu(E)-\frac{2}{n}h^0(E)+2.$
\end{definition}
If $L$ is a line bundle computing $\Cliff(C)$, the equality $\Cliff(\oplus^nL)=\Cliff(L)$ shows that $\Cliff_n(C)\leq \Cliff(C)$.
\begin{definition}
The gonality sequence $\{\delta_r\}_{r\geq 1}$ of $C$ is defined as
\begin{align}\label{gonalityseq}
\delta_r:=\min \{ d: \mbox{$C$ admits a $g^r_d$ }
\}.
\end{align}
\end{definition}
As a well-known and useful fact, the following inequalities hold:
\begin{align}\label{Clifgon}
\delta_1-3\leq \Cliff(C)\leq \delta_1-2.
\end{align}
Since $\Cliff_2(C)\leq \Cliff(C)$ it follows that $\Cliff_2(C)\leq \delta_1-2$, however the inequality $\delta_1-3\leq \Cliff_2(C)$ 
does not hold in general. Recall that for a general curve $C$
$$
\delta_r= \left\lceil  \frac{rg}{r+1}+r \right \rceil.
$$

The following Lemma is a crucial observation in our arguments.
\begin{lemma}\label{Newsteadlem}
  Suppose $E$ is a rank $2$ bundle computing $\Cliff_2(C)$ with $\mu(E)\leq g-1$. 
 If $E$ possesses a line subbundle $M$ such that $h^0(M) \geq 2$, then 
$$\Cliff_2(C)=\Cliff(E)=\Cliff(M)=\Cliff \left( E/M\right)=\Cliff(C).$$
Moreover, the following equality holds
$h^0(E)=h^0(M)+h^0\left(E/M \right)$.
\end{lemma}
\begin{proof}
Given that $E$ computes $\Cliff_2(C)$, this follows directly from \cite[Lemma 2.6]{LNe}.	
\end{proof}

\subsubsection{Lazarsfeld--Mukai Bundles} 
Recall that a coherent system of type $(n, d, n+m)$ on $C$ is a pair $(E, V)$ where $E$ is a vector bundle of rank $n$ and degree $d$ and $V$ is a $(n+m)$-dimensional subspace of  $H^0(E)$. A coherent system $(E, V)$ is called complete if $V=H^0(E)$ and non-complete otherwise. It is called globally generated if the evaluation morphism 
\begin{align}\label{evaluemorphism}
    \phi_{E, V}: V\otimes \mathcal{O}_C \rightarrow E\end{align} 
is surjective.

If $(E, V)$ is a globally generated coherent system, then the kernel $M_{V,E}$ 
of $\phi_{E, V}$ in the exact sequence 
\begin{align}
	0 \rightarrow M_{E,V} \rightarrow V\otimes\mathcal{O}_C \stackrel{\phi_{E, V}}{\longrightarrow} E\rightarrow 0,  
\end{align}
is called the Lazarsfeld--Mukai bundle of $(E, V)$.
 We abbreviate $M_E$ for $M_{E, V}$ when $V=H^0(E)$, and we refer to it as the Lazarsfeld--Mukai bundle of $E$.

\subsubsection{Butler Diagram}
For a globally generated coherent system $(E, V)$ and for a subbundle $S\subseteq M_{E, V}$, there exists a diagram: 

{\small
\begin{align}\label{Butlerdiagram}
\xymatrix{& 0   \ar[d]& 0\ar[d] &\\
 0\ar[r]&S\ar[d] \ar[r],&W\otimes \mathcal{O}_C\ar[d]\ar[r] &F_S \ar[r]\ar[d]^{\alpha_S}&0\\
0\ar[r] &M_{E, V}\ar[r]& V\otimes \mathcal{O}_C\ar[r]& E\ar[r]&0.
}
\end{align}
}
We refer to Diagram \eqref{Butlerdiagram} as the {\it Butler diagram} of $(E, V, S)$, and abbreviate it to the Butler diagram of $(E, S)$
 when $V=H^0(E)$.

The vector subspace $W \subset V$  is determined as 
$$W^*:=\im (V^*\rightarrow H^0(S^*)).$$
Then, $W^*$ generates $S^*$ and $F_S$ is defined as
$$F_S:=[\ker (W^*\otimes \mathcal{O}_C\rightarrow S^*)]^*.$$ 
Throughout this paper, we set: 
\begin{align}\label{kernelalpha}
I_S:=\im \alpha_S, \quad N_S:=\ker \alpha_S
\end{align}
We now record some useful properties of the Buttler diagram. 
For further details, see for instance \cite[Remark 2.2]{MS}.

\begin{Properties}\label{Butler_prop}
\mbox{}
\begin{enumerate}
\item The bundle $F_S$ is globally generated by $W\subseteq H^0(E)$,

\item The map $\alpha_S$ is non-zero,

\item $H^0(F_S^*)=0$,

\item If $S$ is assumed to be a destabilizing subbundle of $E$ with maximal slope and $\rk(F_S)>n$, then 
\begin{align}\label{F_sdegreeinequality1}
\deg(F_S)\leq \frac{\dim W-\rk(F_S)}{\dim W-\rk I_S}\cdot \deg(I_S).
\end{align}
If, in addition, $E$ is semistable, then
\begin{align}\label{F_sdegreeinequality2}\deg (E)\geq \deg (I_S)> \deg F_S.\end{align}
 The last inequality is strict because $\rk F_S > \rk I_S$. 
\end{enumerate}
\end{Properties}
Set $Q:=M_E/S$ in Diagram (\ref{Butlerdiagram}). We will also use the following result from \cite[Theorem 1.1(3)]{CT}.
\begin{lemma}\label{Abel}
 If $S$ is stable of maximal slope, then $H^0(Q)=0$.
\end{lemma}

\subsubsection{ Linear Stability} 
For a generated coherent system $(E, V)$ of type $(n, d, k)$ with $d>0$, the linear slope of $(E, V)$,
 denoted by $\lambda(E, V)$, is defined in \cite{CHE} as $$\lambda(E, V):=\frac{d}{k-n}.$$
When $V=H^0(E)$ we abreviate $\lambda(E, V)$ as $\lambda(E)$.

 Recently, A. Castorena, G. H. Hitching, and E. Luna have extended the notion of linear stability to coherent systems, as developed in \cite{CHE} and \cite{CH2024}. We recall the following definition from \cite{CHE}.

\begin{definition}
A generated coherent system $(E, V)$ of type $(n, d, k)$ is called linearly (semi)stable if for each globally generated coherent subsystem $(F, W)$ of $(E, V)$, with $\deg(F)> 0$, we had 
$$\lambda(F, W)(\geq)> \frac{d}{k-n}\cdot$$
\end{definition}
\begin{remark}
If $M$ is a locally free subsheaf of a vector bundle $E$ then its saturation $M^s\subset E$,  satisfies 
$\mu(M)\leq \mu(M^s).$
It follows that one can equivalently define (semi)stability of vector bundles by checking the slope of their subbundles. However, this argument fails in the realm of linear (semi)stability: although a subsheaf $M\subset E$ may be globally generated, its saturation $M^s$ need not preserve this property.
 \end{remark}
 
Barja--Stoppino \cite{BS} proved that if $V\subset H^0(K_C)$ is a general subspace of codimension $\leq \Cliff(C)$ then $(K_C, V)$ is linearly semistable. They used this fact to study a lower bound for the slope of fibred surfaces. In an analogous manner Mistretta--Stoppino proved the following Proposition in \cite[Proposition 3.3]{MS}:
\begin{proposition}\label{MSprop}
Let $C$ be a curve of genus $g\geq 2$. Let $L\in \Pic(C)$ be a globally
generated line bundle such that $\deg(L)-2(h^0(L)-1)\leq \Cliff(C)$.
 Then $L$ is linearly
semistable. It is linearly stable unless 
$L=K_C(D)$ with $D$ an effective divisor of degree
$2$, or $C$ is hyperelliptic and $\deg(L)-2(h^0(L)-1)$. 
\end{proposition}
\begin{remark}
(i) The equality in Proposition \ref{MSprop} does not imply that $L$ computes the Clifford index of $C$. 
In fact, there exist line bundles $L$ for which $\Cliff(L)=\Cliff(C)$ but $h^1(L)\leq 1$. By definition, such a line bundle $L$ does not contribute to the Clifford index of $C$. For example, consider a general line bundle $L$ of degree $\frac{g-1}{2}$ on a general curve of odd genus $g=2g_1+1$. Then 
$\Cliff(K\otimes L^*)=\Cliff(C)$, but $K\otimes L^*$ does not compute the Clifford index of $C$, because $h^0(L)=0$ and thus it does not contribute to $\Cliff(C)$. However, Proposition \ref{MSprop}  applies to the line bundle $K\otimes L^*$, as it satisfies the inequality stated therein. 

(ii) A similar situation occurs for the line bundles appearing in Lemma \ref{Newsteadlem}. Since $E$ is semistable with 
$\mu(E)\leq g-1$, the line bundle $M$ satisfies $\deg(M)\leq g-1$. Therefore, it contributes to $\Cliff(C)$ and by the result of the mentioned lemma,  it computes $\Cliff(C)$. However, there is no guarantee that the quotient line bundle $E/M$ computes $\Cliff(C)$. Nevertheless, by the result of Lemma \ref{Newsteadlem}, we have $\Cliff(E/M)=\Cliff(C)$. In particular, by \cite[Proposition 3.3]{MS} and \cite[Theorem 5.1]{MS}, we have that $M_{E/M}$ is linearly semistable in general and stable under the hypothesis in \cite[Proposition 3.3]{MS}.
\end{remark}

To conclude, we recall the following conjecture due to Mistretta--Soppino (\cite[Conjecture 6.1]{MS}). 
\begin{conjecture}\label{MSconjecture}
(\textit{MS Conjecture}):	Let $(L,V )$ be a generated linear series as above. If $\deg(L)-2(\dim V-1)\leq \Cliff(C)$, then $(L,V)$ is linearly (semi)stable if and only if $M_{L,V}$ is (semi)stable.
\end{conjecture}
Mistretta–Stoppino proved Conjecture \ref{MSconjecture} under certain conditions, including the case $V=H^0(L)$, which plays a crucial role in our arguments, see \cite[Theorem 5.1]{MS}.
For higher ranks, we expect that at least the following extension of Conjecture (\ref{MSconjecture}) holds.
\begin{conjecture}\label{MSBconjecture}
		Let $E$ be a globally generated vector bundle computing $\Cliff_n(C)$ with $n\geq 2$.  Then\\
	(i) $E$ is linearly (semi)stable. \\
	(ii) $E$ is linearly (semi)stable if and only if $M_{E}$ is (semi)stable.
\end{conjecture}
\begin{remark}\label{rem0}
The globally generated condition in Conjecture \ref{MSBconjecture} is redundant when $n=2$, as Lange and Newstead proved in \cite{LNe1} that every vector bundle computing rank two Clifford index is primitive; that is both $E$ and $K\otimes E^*$ are globally generated.
\end{remark}

\subsubsection{Butler's Conjecture}\label{butlersection}
For a given coherent system $(E, V)$ and a non-negative real number $\alpha$, we say that $(E, V)$ is $\alpha$-(semi)stable, if for each coherent subsystem $(F, W)$ of $(E, V)$ we have $$\mu_\alpha(F,W)(\leq)<\mu_\alpha(E,V),$$ in which $\mu_\alpha(F, W)$ is defined to be 
$$\mu_\alpha(F, W):= \frac{\deg(F)+\alpha\cdot \dim W}{\rk (F)}\cdot$$ 
For $\alpha> 0$, there exists a moduli space $G(\alpha, n, d, n+m)$ parametrizing  $\alpha$-semistable coherent systems of type $(n,d,n+m)$ over $C$. If $(E, V)$ is $\alpha$-stable for 
small values of $\alpha$, then $E$ would be semistable, as well. For $\alpha$ close to $0$,
 following \cite{Ortetal} and \cite{CH2024},   we set  
$$S_0(n,d,n+m) := \{(E,V ) \in G(\alpha, n,d,n+m) : (E,V ) \quad \!\!\!\! \mbox{generated}  \}.$$
The locus $S_0(n, d, n+m)$ is expected to have dimension equal to the Brill--Noether number $\beta(n, d, n+m)$, given by
 $$\beta(n, d, n+m)=n^2(g-1)+1-(n+m)\cdot [(n+m)-d+n(g-1)].$$
 
\begin{definition}
Let $(E, V)$ be a coherent system.
\begin{itemize}

\item A coherent subsystem $(F, W) \leq (E, V)$ is said to be \emph{$\alpha$-destabilizing} if $\mu_\alpha(F, W) \geq \mu_\alpha(E, V)$.

     \item A subbundle $F \leq E$ is said to \emph{destabilize} $E$ if $\mu(F) \geq \mu(E)$. If $E$ is semistable and admits a destabilizing subbundle, then $E$ is referred to as a \emph{strictly semistable} bundle.
    
    \item Similarly, a coherent subsystem $(F, W) \leq (E, V)$ is said to \it{linearly destabilize} $(E, V)$ if $\lambda(F, W) \geq \lambda(E, V)$. If $(E, V)$ is linearly semistable and admits such a subsystem, then it is called a \it{strictly linearly semistable} coherent system.
\end{itemize}
\end{definition}

\begin{conjecture}\label{Butconjecture}
(\textit{D.C. Butler}):
Suppose that $C$ is a general curve and $(E,V)$ a general element of any
 component of $S_0(n,d,n+m)$. Then the coherent system $(M_{E, V}^*, V^*)$ is $\alpha$-stable for $\alpha$ close to $0$, and the map $(E,V ) \mapsto (M_{E, V}^*, V^*)$ gives a birational equivalence between $S_0(n, d, n+m)$ and $S_0(m, d, n+m)$.    
\end{conjecture}
\begin{definition}
 Following \cite{Ortetal} and \cite{CH2024}, we shall say the Butler's conjecture holds non-trivially for type $(n, d, n+m)$, if $S_0(n, d, n+m)$ is nonempty and conjecture \ref{Butconjecture} holds.
\end{definition}

\section{Linear Stability of Bundles computing the second Clifford Index}\label{Linearstabilitysection}
We begin this section with two lemmas that will play a central role in many of the arguments throughout the paper.
\begin{lemma}\label{Newsteadlem1}
Suppose $(E, V)$ is a rank $n$ globally generated coherent system with $\mu(E)\leq g-1$ and 
    $\frac{d_E}{n}-\frac{2}{n}\dim V+2\leq \Cliff(C)$. If $E$ is semistable and $L\subset E$ is an invertible subsheaf, then 
	\begin{align}\label{subsections}
		\dim [H^0(L)\cap V]\leq \frac{\dim V}{n}\cdot
	\end{align}
In particular, if $\Cliff(E)=\Cliff_2(C)$ and there is a line subbundle $L\subset E$ with $h^0(E)=2h^0(L)$, then   $E$ is strictly semistable.	
\end{lemma}
\begin{proof}
Since $\dim V\geq n+1$, the assertion holds whenever $\dim [H^0(L)\cap V]\leq 1$. If $\dim [H^0(L)\cap V]\geq 2$ then, using $\mu(L)\leq \mu(E)\leq g-1$, we obtain $h^1(L)\geq 2$. Thus, $L$ contributes to the Clifford index and 
	we have $$d_L-	2\dim [H^0(L)\cap V]+2\geq d_L-2h^0(L)+2\geq \Cliff(C)\geq \frac{d_E}{n}-\frac{2}{n}\dim V+2.$$
	This, together with the inequality $d_L\leq \frac{d_E}{n}$, which follows from the semistability of $E$, establishes the assertion.
	
For the second part, by Lemma \ref{Newsteadlem}, we have
$$\Cliff_2(C)=\Cliff(E)=\Cliff(L)=\Cliff(C).$$ Therefore, the condition $h^0(E)=2h^0(L)$ implies $\deg(E)=2\deg(L)$, as required.
\end{proof}

\begin{corollary}\label{indicecomputingcor}
Let $C$ be a non-hyperelliptic curve and suppose that the vector bundle $E$ computes $\Cliff_n(C)$. If $L$ is a non-trivially generated  invertible subsheaf of $E$, then
$\lambda(L)\geq \lambda(E)$.   
\end{corollary}
\begin{proof}
Since $L\subset E$ is non-trivial and a globally generated line bundle on $C$, then $h^0(L)\geq 2$. Moreover, as $\mu(E)\leq g-1$, the line bundle $L$ contributes to $\Cliff(C)$. Thus we have $$d_L-2h^0(L)+2\geq \Cliff_n(C)=\frac{d_E}{n}-\frac{2}{n}h^0(E)+2.$$
This is equivalent to 
$$\frac{d_L}{h^0(L)-1}-2\geq \frac{d_E-2(h^0(E)-n)}{n\cdot (h^0(L)-1)}\cdot$$
Therefore, it suffices to prove $$\frac{d_E-2(h^0(E)-n)}{n\cdot (h^0(L)-1)}\geq \frac{d_E}{h^0(E)-n}-2,$$ which is equivalent to $ n\cdot h^0(L) \leq h^0(E).$
Thus, $\lambda(L)\geq \lambda(E)$ by Lemma \ref{Newsteadlem1}. 
\end{proof}
\begin{lemma}\label{LNlem}
Let $F$ be a globally generated vector bundle of rank two on $C$, which is not semistable and satisfies $\mu(F)\leq g-1$. Assume further that $F$ admits no trivial quotient line bundle. Then
$$\Cliff(F)\geq \Cliff(C).$$
\end{lemma}
\begin{proof}
	Let $0\rightarrow L_1\rightarrow F\rightarrow L_2 \rightarrow 0$ be a Harder--Narasimhan filtration of $F$ with $\deg(L_1)>\mu(F)>\deg(L_2)$. The assumption of non-semistability implies that $F$ is non-trivially generated. 
	Since $F$ admits no trivial quotient by assumption, $L_2$ can not be the trivial line bundle; therefore, it is non-trivially generated; hence $h^0(L_2)\geq 2$. As $\deg(L_2)< g-1$, the line bundle $L_2$ contributes to $\Cliff(C)$. 
    
    From $\deg(L_1)>\deg(L_2)$ we conclude that if $h^0(L_1)\leq h^0(L_2)$ then $\Cliff(L_1)>\Cliff(L_2)$.
    
	Likewise; if $h^1(L_1)\leq h^0(L_2)$ then from 
	$$\deg(K\otimes L_1^*)=(2g-2-d_F)+\deg(L_2)\geq \deg(L_2),$$
	we get $\Cliff(K\otimes L_1^*)\geq \Cliff(L_2)$.
	Now, since $\Cliff(L_1)=\Cliff(K\otimes L_1^*)$, we have $\Cliff(L_1)\geq\Cliff(L_2)$. 
	
	If neither of the two cases occur, then $L_1$ contributes to $\Cliff(C)$.
Hence, the result follows from
	$$\Cliff(F)\geq \frac{\Cliff(L_1)+\Cliff(L_2)}{2}\geq \Cliff(C).$$
\end{proof}

\begin{remark}\label{remark2}
Let $F$ be as in Lemma \ref{LNlem}, and $(E, V)$ as in Lemma \ref{Newsteadlem1}.
	\begin{itemize}
	\item 	If $F$ admits a trivial quotient, then it's Clifford index can be strictly smaller than $\Cliff_2(C)$. For instance, suppose $C$ is a general curve. Then, there exists an integer $t$ with 
    $$\delta_2\leq t < 2\bigg[  \frac{g-1}{2}\bigg]+4 .$$ 
If $L\in W^2_{t}$ is globally generated, then the bundle $F:=L\oplus \mathcal{O}_C$ is also globally generated and satisfies
	$$\Cliff(F)<\Cliff(C)=\Cliff_2(C).$$
	
	\item If $\Cliff_2(C)=\Cliff(C)$, then $V=H^0(E)$ would be the only subspace of $H^0(E)$ satisfying the conditions in  Lemma \ref{Newsteadlem1}. 
	If $\Cliff_2(C)\leq \Cliff(C)-1$ and $E$ is a bundle computing $\Cliff_2(C)$, then every hyperplane $V\subset H^0(E)$ fulfills the assumptions of Lemma \ref{Newsteadlem1}. See \cite{LNe23} and \cite{LNe24} for examples of curves with this property.

    \item
 A surjection $E\to \mathcal{O}_C \to 0$ will split by \cite[Lemma 1.1]{Ba}. Therefore, if $E$ admits a trivial quotient, the trivial quotient bundle would be a direct summand of $E$; however, this fact is not required in Lemma (\ref{LNlem}).  
\end{itemize}

\end{remark}

	\begin{proposition}\label{indicecomputinglem1}
	Suppose $C$ is a non hyperelliptic curve, and let $E$ be semistable vector bundle with $\mu(E)\leq g-1$. Assume $(E, V)$ is a globally generated rank-two coherent system such that
	$$\frac{d_E}{2}-\dim V+2\leq \Cliff(C).$$
    If $(F, W)$ is a globally generated subsystem of $(E, V)$, where $F\subset E$ is non-semi-stable, rank two locally free subsheaf of $E$, then
	$\lambda(F, W)\geq \lambda(E, V)$.
	
Moreover, if $E$ computes $\Cliff_2(C)$ and $V=H^0(E)$, then 	$\lambda(F)> \lambda(E)$ unless either $\Cliff(F)=\Cliff_2(C)$ or $F$ admits a globally generated line subbundle $L$ with $h^0(L)=h^0(F)$ and $\deg(L)=\deg(F)$.
\end{proposition}
\begin{proof}
	If $F$ admits no trivial quotient, then by Lemma (\ref{LNlem}), we have 
	$\Cliff(F)\geq \Cliff(C)$.
	Therefore,
	$$d_F-2(\dim W-2)\geq d_F-2(h^0(F)-2)\geq 2\Cliff_2(C)\geq d_E-2(\dim V-2),
	$$
which implies
	$$\frac{d_F-2(\dim W-2)}{2(\dim W-2)}\geq \frac{d_E-2(\dim V-2)}{2(\dim W-2)}\geq \frac{d_E-2(\dim V-2)}{2(\dim V-2)}.$$
Thus, 
	$$\frac{1}{2}\lambda(F, W)-1\geq \frac{d_E-2(\dim V-2)}{2(\dim W-2)}> \frac{d_E-2(\dim V-2)}{2(\dim V-2)}=\frac{1}{2}\lambda(E, V)-1,$$
	as desired.
	
	If $F$ has a representation $0\to L\to F\to \mathcal{O}_C\to 0$, then 
	$$\lambda(F, W)\geq \lambda(L, H^0(L)\cap W).$$
	Since
	$h^0(L)\geq 2$ and $\mu(L)\leq \mu(E)\leq g-1$, the line bundle $L$ contributes to $\Cliff(C)$. Therefore, we have 
	$$d_L-2\dim [W\cap H^0(L)]+2\geq d_L-2h^0(L)+2\geq \Cliff_2(C)\geq \frac{d_E}{2}-\dim V+2, $$
	which implies 
	$$ \frac{d_L}{\dim W\cap H^0(L)-1}-2\geq \frac{d_E-2(\dim V-2)}{2(\dim W\cap H^0(L)-1)}$$
	Therefore, it suffices to prove 
	$$\frac{d_E-2(\dim V-2)}{2(\dim W\cap H^0(L)-1)}\geq \frac{d_E}{\dim V-2}-2,$$
which is equivalent to 
$$2\dim [W\cap H^0(L)] \leq \dim V.$$
	Thus, by Lemma (\ref{Newsteadlem1}), $\lambda(L, W\cap H^0(L))\geq \lambda(E, V)$.

The second statement follows immediately from our argument.
\end{proof}
Now, we turn to complete the proof of conjecture \ref{MSBconjecture}(i) in the case $n=2$.
Recall that any such bundle $E$, satisfies the following inequality by definition: 
\begin{align}\label{Cliffdefinequality}
d_E\leq 2h^0(E)+2\delta_1-8.
\end{align}
\begin{theorem}\label{indicecomputingthm1}
	Let $C$ be a non-hyperelliptic curve, and
	assume that $E$ computes $\Cliff_2(C)$ satisfying $5\leq h^0(E)$. Then, $E$ is linearly semistable.

	The bundle $E$ fails to be linearly stable if and only if either $E$ admits a rank two globally generated and locally free subsheaf $T$ with $\Cliff(T)=\Cliff_2(E)$, or $E$ contains a line subbundle $L$ with $h^0(L)=\frac{h^0(E)}{2}$.	
\end{theorem}
\begin{proof}
Taking Proposition (\ref{indicecomputinglem1}) and Corollary (\ref{indicecomputingcor})
into account, it suffices to exclude the possibility that globally generated rank two locally free subsheaves of $E$, which are semistable, destabilize 
$E$ linearly.

Suppose $F\subset E$ is a globally generated, locally free, and semistable subsheaf of rank two. Then $h^0(F)\geq 3$ and $\mu(F)<\mu(E)\leq g-1$. We consider two cases:
	
First, assume that $h^0(F)\geq 4$. Then $F$ contributes to $\Cliff_2(C)$, and hence	$\Cliff_2(F)\geq \Cliff_2(E)$. Therefore, $\frac{d_F}{2}-(h^0(F)+2)\geq \frac{d_E}{2}-(h^0(E)-2)$,
 from which we it follows that
 $$\frac{d_F-2(h^0(F)-2)}{2(h^0(F)-2)}\geq \frac{d_E-2(h^0(E)-2)}{2(h^0(F)-2)} \geq \frac{d_E-2(h^0(E)-2)}{2(h^0(E)-2)}\cdot$$
Thus,
	$\lambda(F)\geq \lambda(E),$ as desired.
	
If $h^0(F)=3$, then 
$h^0(\det(F))\geq 2$, so $d_F\geq \delta_1$. Therefore the inequality $\lambda(F)=d_F\leq \lambda(E)$ together with (\ref{Cliffdefinequality}), implies 
$$\delta_1\leq \frac{2h^0(E)+2\delta_1-8}{h^0(E)-2}\comma$$
which is equivalent to
$$ \delta_1\cdot(h^0(E)-4)\leq 2(h^0(E)-4).$$
This is impossible since $h^0(E)\geq 5$ and $C$ is non-hyperelliptic. 

The second statement holds by our argument in this theorem and in the proof of Proposition (\ref{indicecomputinglem1}).
\end{proof}

\begin{theorem}\label{indicecomputingthm0}
	Let $E$ be a vector bundle computing $\Cliff_2(C)$ satisfying $h^0(E)=4$. Then, 
	$E$ is linearly semistable, and it is linearly stable if and only if $E$ does not admit any line subbundle $L$ with $h^0(L)=2$.
\end{theorem}
\begin{proof}
	
We have $d_E\leq 2\delta_1$. By Lemma (\ref{Newsteadlem1}) and semistability of $E$, the only possiblity for a non-trivial line subbundle $L$ of $E$ to be globally generated is $h^0(L)=2$ and $\deg(L)=\delta_1$. So $\lambda(L)=\lambda(E)$ holds for any globally generated invertible subsheaf $L\subset E$.  
	
If $F$ is a globally generated locally free subsheaf of $E$ satisfying $\rk F=2$ and $\deg(F)>0$, then the inequality $\deg(F)\geq \delta_1$ holds by $h^0(F)=3$. From $\delta_1\leq \deg(F)=\lambda(F)$ and $\lambda(E)\leq \delta_1$, we conclude that
$\lambda(F)\geq \lambda(E)$. Furthermore, the equality $\lambda(F)=\lambda(E)$ holds if and only if $F$ has a representation either as 
$0\rightarrow \mathcal{O}_C\rightarrow F\rightarrow L \rightarrow 0$ or as $ 0\rightarrow L\rightarrow F\rightarrow  \mathcal{O}_C \rightarrow 0$.

 The sequence in the first case must induce an exact sequence on global sections, since $h^0(F)=3$, and we have  $h^0(\mathcal{O}_C)+h^0(L)\leq 3$. However, the map $H^0(L)\otimes H^0(L)\rightarrow H^0(K\otimes L)$ is surjective by the base point free pencil trick. Therefore, the only possibility is $F=\mathcal{O}_C\oplus L$. In this case, again $E$ has a line subbundle as stated.

In the second case $L$ is a subsheaf of $E$ and is actually a line subbundle. Therefore, $\lambda(L)=\lambda(E)$.
\end{proof}
\begin{corollary}
Let $C$ be a non-hyperelliptic curve, and
suppose $E$ computes $\gamma_2(C)$, with $\gamma_2(C)$ as in Definition (\ref{Cliffordindices}). 
 Then, $E$ is linearly semistable.
Furthermore, the bundle $E$ fails to be linearly stable if and only if $E$ has a rank two locally free subsheaf  $T$ with $\Cliff(T)=\Cliff_2(E)$, or $E$ contains a line subbundle $L$ with $h^0(L)=\frac{h^0(E)}{2}$.	
\end{corollary}
\begin{proof}
If $h^0(E)\geq 4$, then $E$ computes $\Cliff_2(C)$ as well. Therefore, the result follows from Theorems (\ref{indicecomputingthm1}) and (\ref{indicecomputingthm0}). 

If $h^0(E)=3$, then $M_E$ is a line bundle and thus stable. Therefore, $E$ is linearly stable. 
\end{proof}

\begin{remark}
If $C$ is hyperelliptic, and $E$ is semistable with $\Cliff(E)=0$, then $E=mg^1_2\oplus mg^1_2$, for $1\leq m< g-1$, by \cite[Proposition 2]{Re}. So, $E$ is linearly semistable, but not linearly stable. 
\end{remark}

\section{Conjecture \ref{MSBconjecture} in two special cases}\label{equivalencesection}
\subsubsection{Bundles admitting a subpencil}

In this subsection, we establish 
the (MS) conjecture for rank two bundles that admit a line subbundle $L$ with $h^0(L)=2$. 

\begin{definition}
	A bundle $E$ is said to admit a subpencil if there exists a line subbundle $L\subset E$ such that $h^0(L)=2$.
\end{definition}

The following lemma is a restatement of \cite[Lemma 4.3]{MS} for higher ranks, and we omit its proof.
\begin{lemma}\label{finalthrm}
	Suppose $E$ computes $\Cliff_n(C)$ 
 and for $S\subset M_E$ the exact sequence 
 \begin{align}\label{exactsequence}
 	0\rightarrow \oplus^{r_{F_S}-1}\mathcal{O}\rightarrow F_S \rightarrow \det F_S\rightarrow 0,
 	 \end{align}
 induces an exact sequence on global sections.
  Assume $\deg(E)\leq \delta_1\cdot(h^0(E)-n)$ and $\rk(F_S)\geq 2$. Then $\mu(S)\leq \mu(M_E)$ and the equality holds if and only if 
 \begin{itemize}
\item $W=H^0(F_S)$, with $W$ as in Diagram \eqref{Butlerdiagram}, and 
 	
\item $\delta_1=\frac{\deg(F_S)}{h^0(\det F_S)-1}=\frac{\deg(E) }{h^0(E)-n}$. 	
 \end{itemize}
\end{lemma}
\begin{theorem}\label{specialcase}
Let $E$ compute $\Cliff_2(C)$ and $S\subset M_E$
is of maximal slope with $\rk F_S\geq 2$.  If $S$ destabilizes $M_E$, that is $\mu(S)\geq \mu(M_E)$, 
and $\rk(I)=1$, where $I:=\im(\alpha_S)$, then the sequence (\ref{exactsequence}) induces exact sequence on global sections. Moreover, $\deg(F_S)\geq \delta_1\cdot(h^0(\det(F_S))-1)$. 
\end{theorem}
\begin{proof}
The proof proceeds along the same lines as the argument in Theorem 5.1 of \cite{MS}.

	By semistability of $E$, we have $\deg(I)\leq \frac{\deg(E)}{2}$. Since $I\neq E$, it follows that $M_{I, W}$ is a proper subbundle of $M_E$. So $\mu(M_{I, W})\leq \mu(S)$, and consequently,  $\deg(F_S)\leq \deg(I)$. Therefore, 
	\begin{align}\label{inequality2}
		\deg(F_S)\leq\frac{\deg(E)}{2}\leq g-1.
	\end{align}
Hence $\det(F_S)$ contributes to $\Cliff(C)$, because it is non-trivially generated. Thus,
	\begin{align}\label{inequality3}
		\deg(\det(F_S))-2h^0(\det(F_S))+2\geq \Cliff(C)\geq \frac{\deg(E)}{2}-h^0(E)+2=\Cliff_2(C).
	\end{align}
	This, by (\ref{inequality2}) implies that $h^0(E)\geq 2\cdot h^0(\det(F_S)).$ In particular, we have
	\begin{align}\label{inequality4}
		\frac{h^0(\det(F_S))-1}{h^0(E)-2}\leq\frac{1}{2}\cdot
	\end{align}
	Now, we prove that the exact sequence 
\begin{align}\label{exactsequence3}
0\rightarrow \oplus^{r_S-1}\mathcal{O}_C\rightarrow F_S\rightarrow \det F_S\rightarrow 0
\end{align}
 induces an exact sequence on global sections. Otherwise, we would have
	$$h^0(\det F_S)-1>h^0(F_S)-\rk(F_S)\geq \rk S  \geq \deg (\det F_S))\cdot\frac{h^0(E)-2}{\deg(E)}\cdot$$
	From this, and since we have $\frac{\deg(E)}{h^0(E)-2}=2+\frac{2\Cliff_2(C)}{h^0(E)-2}$ by definition of $\Cliff_2(C)$, it follows that,
	$$\deg(\det F_S)<\frac{\deg(E)}{h^0(E)-2}\cdot(h^0(\det F_S)-1)=(h^0(\det F_S)-1)\left(2+\frac{2\Cliff_2(C)}{h^0(E)-2}\right)=$$
	$$2(h^0(\det F_S)-1)+2\cdot\left(\frac{h^0(\det F_S)-1}{h^0(E)-2}\right)\cdot \Cliff_2(C).$$
	In combination with inequality (\ref{inequality4}), this yields 
\begin{align}\label{degF_S}
\deg(\det F_S)<2(h^0(\det F_S)-1)+\Cliff_2(C).    
\end{align}
	This contradicts (\ref{inequality3}). Therefore, the exact sequence (\ref{exactsequence3}) induces an exact sequence on global sections.

By the  exactness of the sequence of global sections of (\ref{exactsequence3}), last assertion follows  from \cite[Prop 4.2]{MS}.
\end{proof}

\begin{remark}\label{remark3}
Recall that if $E$ computes $\Cliff_2(C)$ then, since $h^0(E)\geq 4$, we have
$$\deg(E)\leq 2\delta_1-4+2h^0(E)-4\leq 2(\delta_1-2)\cdot \left (\frac{h^0(E)-2}{2}\right)+2h^0(E)-4=\delta_1\cdot (h^0(E)-2),$$	
with equality holding if and only if either $C$ is hyperelliptic and $d_E=2(h^0(E)-2)$, or $h^0(E)=4$.
\end{remark}

\begin{corollary}\label{rank1cor}
Let $E$ compute $\Cliff_2(C)$ and $S\subset M_E$
is of maximal slope. If $E$ is linearly stable, then
we have $\mu(S)< \mu(M_E)$. 
\end{corollary}
\begin{proof}
If $\rk(F_S)=1$, then the assertion follows directly from the definition of linear stability. If $\rk(F_S)\geq 2$, and $\mu(S)\geq \mu(M_E)$, then the only possibility--by
Lemma \ref{finalthrm}, Theorem \ref{specialcase} and Remark \ref{remark3}--is that $\mu(S)=\mu(M_E)$, in which case, 
 $$\deg(E)=\delta_1\cdot (h^0(E)-2).$$ 	
Remark \ref{remark3} then implies that 
either $C$ is hyperelliptic and $d_E=2(h^0(E)-2)$, or $h^0(E)=4$.
If $h^0(E)=4$, then $\overline{I}_S$, the saturation of $I_S$,  is a line subbundle of $E$. Since $h^0(I_S)\geq 2$ we also have $h^0(\overline{I}_S)\geq 2$. Lemma \ref{Newsteadlem1} implies that $h^0(\overline{I}_S)\leq 2$, and so $h^0(\overline{I}_S)=2$. As $E$ is semistable we have $\deg(\overline{I}_S)\leq \frac{d_E}{2}$. Moreover, since $E$ is linearly semistable, results from the previous section, imply that $\deg(\overline{I}_S)=\lambda(\overline{I}_S)\geq\lambda(E)$. Summarizing we obtain $\lambda(\overline{I}_S)=\lambda(E)$, which contradicts the assumption that $E$ is linearly stable.

If $C$ is hyperelliptic and $d_E=2(h^0(E)-2)$, then by \cite[Proposition 2]{Re} we have $E=g^1_2\oplus g^1_2$, which again contradicts the  linear stability of $E$.
\end{proof}
\begin{lemma}\label{Newsteadlem3}
	Let $E$ be a rank two vector bundle computing $\Cliff_2(C)$, and let $L\subset E$ be a line subbundle with $h^0(L)\geq 2$. Then, we have 
	$$\lambda(L)\geq \lambda(E)\geq \lambda(G),$$
	 where $G:=\frac{E}{L}$.
\end{lemma}
\begin{proof}
	The inequality $\lambda(L)\geq \lambda(E)$ is a direct consequence of Theorem \ref{indicecomputingthm1} and Theorem \ref{indicecomputingthm0}.
	
	Since, by Lemma \ref{Newsteadlem}, we have $\Cliff(G)=\Cliff(E)=\Cliff(C)$, it follows that
	$$\lambda(G)=\frac{\Cliff(C)}{h^0(G)-1}+2, \quad \lambda(E)=\frac{2\Cliff(C)}{h^0(E)-2}+2.$$
	Therefore, $\lambda(G)\leq \lambda(E)$ is equivalent to 
	\begin{align}\label{thirdmainlemma}
		\frac{\Cliff(C)}{h^0(G)-1}\leq \frac{2\Cliff(C)}{h^0(E)-2}\comma
	\end{align}
	which is immediate whenever $\Cliff(C)=0$. Otherwise,  by Lemma (\ref{Newsteadlem1}), we obtain
    $$2h^0(L) \leq  h^0(E),$$ and  \eqref{thirdmainlemma} 
    follows from Lemma \eqref{Newsteadlem}.
\end{proof}
\begin{proposition}\label{diagramlem}
Let $(E, V)$ be a globally generated coherent system 
and
\begin{align}\label{propexactseq1}
    0\rightarrow (F_1, V_1)\rightarrow (E,V) \rightarrow (F_2, V_2)\rightarrow 0  \end{align}
 an exact sequence of globally generated coherent systems induced from an exact sequence 
\begin{align}\label{bundexseq}
0\rightarrow F_1\stackrel{i}{\longrightarrow} E \stackrel{\pi}{\longrightarrow} F_2 \rightarrow 0.
\end{align}
 Suppose a subbundle $S\subset M_{E, V}$ fits into a short exact sequence 
\begin{align}\label{seq1}
0\rightarrow 
	S_1\rightarrow S \rightarrow S_2\rightarrow 0,
	\end{align}
where $S_i\subset M_{F_i, V_i}$ are subbundles and $I= \im \alpha_S$,  $I_i= \im \alpha_{S_i}$, $i=1,2$ where $\alpha_{S_i}$ are the corresponding morphisms in the Butler diagram \eqref{Butlerdiagram}. Then,

\noindent (i) If $S_2=0$ then 
$F_S=F_{S_1}$ and $ I\cong I_1$.
    
\noindent (ii) If $S_1=0$ then 
$F_S\cong F_{S_2}$ and $ I\cong I_2$.

\noindent (iii) If $S_i\neq 0$ and either the morphism $H^0(S^*)\rightarrow H^0(S_1^*)$ is surjective or $\dim V_1=2$, then we obtain an exact sequence
	$$0\rightarrow F_{S_1}\rightarrow F_S \rightarrow F_{S_2}\rightarrow 0,$$
together with the commutative diagram:
{\small
	\begin{align}\label{relativeButlerdiagram}
		\xymatrix{
			0\ar[r]&F_{S_1} \ar[d]^{\alpha_{S_1}} \ar[r],&F_S\ar[d]^{\alpha_S} \ar[r] &F_{S_2} \ar[d]^{\alpha_{S_2}}\ar[r]&0\\
			0\ar[r] &F_1\ar[r]^{i}& E\ar[r]^{\pi}& F_2\ar[r]&0.
		}
	\end{align}
}
\end{proposition}
\begin{proof}
  Let  
\begin{align}\label{vectexseq}
0\rightarrow V_1\stackrel{\overline{i}}{\longrightarrow} V \stackrel{\over{\pi}}{\longrightarrow} V_2 \rightarrow 0
\end{align}
be the exact sequence of vector spaces, induced by the exact sequence (\ref{propexactseq1}).
The exact sequence \eqref{seq1} then gives rise to the following commutative diagram 
{\small
	\begin{align}\label{relativeButlerdiagram0}
	\xymatrix{
		0\ar[r]&V_2^* \ar[d]^{\theta_2} \ar[r],&V^*\ar[d]^\theta \ar[r] &V_1^* \ar[d]^{\theta_1}\ar[r]&0\\
		0\ar[r]&H^0(S_2^*)\ar[r]^{\gamma_2}&H^0(S^*)\ar[r]^{\gamma_1}&H^0(S_1^*).&
	}
\end{align}
}
From this we obtain an exact sequence $0\rightarrow W_2^* \rightarrow W^* \rightarrow 
W_1^*$, where $W^*:=\im \theta$, $W_i^*:=\im \theta_i$ for $i=1, 2$. Consequently, we obtain the following commutative diagram: 
{\small
	\begin{align}\label{relativeButlerdiagram21}
	\xymatrix{& 0   \ar[d]& 0\ar[d] &0\ar[d]\\
			0\ar[r]&S_1\ar[r]\ar[d]&S\ar[r]\ar[d]&S_2\ar[d]\ar[r] &0\\
		& W_1\otimes \mathcal{O}_C\ar[r]\ar[d]^{\eta_1^*}& W\otimes \mathcal{O}_C\ar[r]\ar[d]^{\eta^*}& W_2\otimes \mathcal{O}_C\ar[d]^{\eta_2^*}\ar[r] &0\\
	& F_{S_1}\ar[r]^{h_1}\ar[d]& F_{S}\ar[r]^{h_2}\ar[d]& F_{S_2}\ar[d]\ar[r] &0 \\
		& 0   &0& 0. 
	}
\end{align}
}\\

\noindent \textit{Proof of (i):}
 If $S_2=0$, then $\gamma_1$ in  Diagram \eqref{relativeButlerdiagram0} is injective.
So its restriction to $W^*$ gives an isomorphism from $W^*$ to $W_1^*$.  As $S=S_1$, we obtain $F_S=F_{S_1}$. The isomorphism $I\cong I_1$ is immediate by the injectivity of $i$ in the following commutative square 
{\small
 	\begin{align}\label{relativeButlerdiagram211i}
 	\xymatrix{
 		& F_{S_1}\ar[r]^{=}\ar[d]^{\alpha_{S_1}}& F_{S}\ar[d]^{\alpha_S}\\
 		& F_1\ar[r]^{i}& E.
 	}
 \end{align}
}


\noindent \textit{Proof of (ii):} If $S_1=0$, then $\theta_1$ in Diagram \eqref{relativeButlerdiagram0} vanishes. In particular, $W_1^*=0$. Thus for any $v\in V^*$, we have 
\begin{align}\label{restriction}
    \theta(v)=\theta(v_1+v_2)=\theta_2(v_2),
\end{align}
where $v_1\in V_1^*$ and $v_2\in V_2^*$.
 Hence, $W\cong W_2$. From the third row of Diagram \eqref{relativeButlerdiagram21}, it follows that $F_S\cong F_{S_2}$. 

Now, in order to prove $I\cong I_2$, 
it suffices to prove that $\ker\alpha_S \cong \ker\alpha_{S_2}.$
Recall that $\alpha_S$ is defined as 
$$\alpha_S:= (\gamma_{E,V}\circ \phi_{E,V}^*)^*,$$ where $\phi_{E,V}^*$ is the dual of the evaluation morphism  \eqref{evaluemorphism}, and $\gamma_{E,V}$ is the dual of $W\otimes \mathcal{O}_C \rightarrow V \otimes \mathcal{O}_C$ making the diagram
{\small
	\begin{align}\label{relativeButlerdiagram211}
	\xymatrix{
	0\ar[r]	& F_S^*\ar[r]& W^*\otimes \mathcal{O}_C\ar[r]&S^*\ar[r] &0\\
	0\ar[r]& E^*\ar[u]\ar[r]^{\phi_{E,V}^*}& V^*\otimes\mathcal{O}_C\ar[r]\ar[u]^{\gamma_{E,V}}& M_{E, V}^*\ar[u]\ar[r] &0.
}
\end{align}
}
commutative. Here we identify $F_S^*$ with its image in  $W^*\otimes \mathcal{O}_C$.

 Similarly $\alpha_{S_2}=(\gamma_{E,V}\circ \bar{\pi}\circ \phi_{F_2,V2}^*)^*$, where $\bar{\pi}$ is the induced morphism in 
 \eqref{vectexseq}.
 Therefore, the proof will be complete whenever we show 
 $$\im (\gamma_{E,V}\circ \phi_{E,V}^*) \cong \im (\gamma_{E,V}\circ \bar{\pi}\circ \phi_{F_2,V2}^*).$$
  Likewise, by the commutative square 
{\small
 	\begin{align}\label{square2}
 	\xymatrix{
 		& E^*\ar[r]^{\phi_{E, V}^*}& V^*\otimes \mathcal{O}_C\\
 		& F_2^*\ar[r]_{\phi_{F_2, V_2}^*}\ar[u]^{\pi^*}& V^*_2\otimes\mathcal{O}_C\ar[u]^{\overline{\pi}},
 	}
 \end{align}
}
it suffices to prove 
$$\im(\gamma_{E,V}\circ \phi_{E,V}^*)\cong\im(\gamma_{E,V}\circ \phi_{E,V}^*\circ \pi^*).$$ As well, we have $V\otimes\mathcal{O}_C=(V_1\otimes\mathcal{O}_C)\oplus(\overline{V_2}\otimes\mathcal{O}_C)$ with $V_2\cong \overline{V_2}\leq V$. 
 
  Consider the following commutative diagram 
{\small
	\begin{align}\label{relativeButlerdiagram212}
	\xymatrix{&0\ar[d]  & 0\ar[d] &\\
			0\ar[r]&F_2^*\ar[r]^{\pi^*}\ar[d]^{\phi_{F_2, V_2}^*}&E^*\ar[r]^{i^*}\ar[d]^{\phi_{E, V}^*}&F_1^* \ar[r]&0\\
	0\ar[r]	& V_2^*\otimes \mathcal{O}_C\ar[r]\ar[d]^{\gamma_{F_2, V_2}}& (\overline{V}_2^*\otimes \mathcal{O}_C)\oplus (V_1^*\otimes \mathcal{O}_C)=\ar[d]^{\gamma_{E,V}} V^*\otimes \mathcal{O}_C& &\\
& W^*\otimes\mathcal{O}_C\ar[r]^=\ar[d]& W^*\otimes\mathcal{O}_C\ar[d]& &. \\
		&   0 & 0, & 
	}
\end{align}
}
and observe that 
\begin{align}\label{1}
(\phi_{E, V}^*)^{-1}(\overline{V}_2\otimes\mathcal{O}_C)=\pi^*(F_2^*).\end{align}
Recall by (\ref{restriction}) that we have 
\begin{align}\label{2}
    \gamma_{E, V}(V^*\otimes \mathcal{O}_C)=\gamma_{E, V}(\overline{V}_2^*\otimes \mathcal{O}_C).\end{align}
This together with the commutativity of the diagram (\ref{relativeButlerdiagram212}), gives 
$$\gamma_{E, V}\circ\phi_{E, V}^*(E^*)=\gamma_{F_2, V_2}\circ\phi_{F_2, V_2}^*(F_2^*),$$
as required.\\
 
\noindent \textit{Proof of (iii):} 
If $\gamma_1$  in Diagram \eqref{relativeButlerdiagram0} is surjective,  the second and third rows  in Diagram \eqref{relativeButlerdiagram21}  are exact on the left. Hence, we obtain Diagram \eqref{relativeButlerdiagram}.

When $\dim V_1=2$, the bundle $F_{S_1}$ is a line bundle. In the diagram:
{\small
 	\begin{align}\label{relativeButlerdiagram211}
 	\xymatrix{
 		& F_{S_1}\ar[r]^{h_1}\ar[d]^{\alpha_{S_1}}& F_{S}\ar[d]^{\alpha_S}\\
 	0 \ar[r]	& F_1\ar[r]^{i}& E,
 	}
 \end{align}
}
the composition morphism $i\circ \alpha_{S_1}$ is nonzero. It follows that the morphism   $h_1$ must also be nonzero. 

 Since $F_{S_1}$ is a line bundle, any nonzero morphism from it to a torsion-free sheaf is injective. Thus, $h_1$ is injective, and the desired commutative diagram is obtained.
\end{proof}

Now, we complete the proof of conjecture \ref{MSBconjecture}(ii) for bundles that compute $\Cliff_2(C)$ and admit a subpencil. 
\begin{theorem}\label{MSconjecturerank=2}
Let $E$ be a vector bundle of rank two and degree $d$, admitting a line subbundle $L\subset E$ with $h^0(L)=2$. Suppose further that $\Cliff(E)=\Cliff_2(C)$. Then 
$E$ is linearly (semi)stable if and only if the Lazarsfeld--Mukai bundle $M_E$
is (semi)stable.
Moreover, if $C$ is non-hyperelliptic, then $M_E$ is stable if and only if $h^0(E)>4$.
If $C$ is hyperelliptic, then $M_E$  is strictly semistable if and only if $E$
 is an extension of the form
$$0\rightarrow g^1_2 \rightarrow E\rightarrow tg^1_2\rightarrow 0. $$ 
\end{theorem}

\begin{proof}
If $M_E$ is (semi)stable, then $E$ is trivially linearly (semi)stable.	
Conversely, we suppose that $E$ is linearly (semi)stable and prove that $M_E$ is (semi)stable. Notice that if $H\subset E$ is 
a globally generated invertible subsheaf of $E$ such that $\lambda(H)=\lambda(E)$, then we have $\mu(M_H)=\mu(E)$ and $M_E$ would be strictly semistable. We now assume that $E$ is linearly stable and aim to show that $M_E$ is stable.

Let $G:= E/L$ and $S$ be a subbundle of $M_E$ of maximal slope.
	If $\rk F_S=1$, then $F_S$ is a line subbundle of $E$, and 
    the linear stability of $E$ implies $\mu(S)<\mu(M_E)$. 
	
	Assume $\rk F_S\geq 2$ and consider the exact sequence 
	$$0\rightarrow S_1 \rightarrow S \rightarrow S_2\rightarrow 0, $$
 where $S_1\subseteq M_L$, $S_2\subseteq M_G$.  Since $S_2$ is a subsheaf of the trivial bundle $H^0(G)\otimes \mathcal{O}_C$, it can not be a torsion sheaf.
Furthermore, note that $S_2\neq 0$. otherwise, given that $h^0(L)=2$, we would have $\rk(F_S)=1$. Therefore, we distinguish between two cases:\\

	\textsf{Case (i):} 
	If $S_1=0$, then 
 by Proposition \ref{diagramlem}, we have 
 $\rk \alpha_S=\rk \alpha_{S_2}=1$. 
 Corollary \ref{rank1cor} implies that 
 $\mu(S)\leq \mu(M_E)$.
 
	 If $\mu(S)=\mu(M_E)$, then by Lemma \ref{finalthrm},  $\deg(E)=\delta_1\cdot (h^0(E)-2)$.  By Remark \ref{remark3}, this implies that either $C$ is hyper elliptic and $d_E=2(h^0(E)-2)$, or $h^0(E)=4$.
	  In the first case, we have
	 $E=g^1_2\oplus g^1_2$ by \cite[Proposition 2]{Re}. Therefore, $E$ is linearly semistable. \\
If $h^0(E)=4$, then $h^0(G)=2$, and hence $\lambda(L)=\lambda(E)$ as desired.
	 \\	
       
	\textsf{Case (ii):} 
Let $\rk S_1=1$ and $\rk S_2\geq 1$. We may assume, without loss of generality, that $S_1=M_L$. Since $h^0(L)=2$, Proposition \ref{diagramlem} yields the exact sequence
$$0\rightarrow F_{S_1}\rightarrow F_S \rightarrow F_{S_2}\rightarrow 0,$$
where $F_{S_1}=L$.

Let $d_1:=\deg(M_L)$, $f_2:=\deg(S_2)$ and $r_2:=\rk (S_2)$. 
 Then the slope inequality
 $$\mu(S)=\frac{d_1+f_2}{1+r_2}\leq \mu(M_E)=\frac{d_1-d_G}{1+\rk(M_G)},$$
is equivalent to
$$d_1\cdot (\rk(M_G)-r_2)+f_2\cdot \rk(M_G)\leq  -d_G-f_2-r_2\cdot d_G.$$

According to \cite[Proposition 3.3]{MS} and \cite[Theorem 5.1]{MS}, the bundle $M_G$ is semistable, so $\mu(S_2)\leq \mu(M_G)$, that is,
 $$
 f_2\cdot \rk(M_G)\leq -r_2\cdot d_G.
 $$
Thus, it suffices to show
	$$d_1\cdot (\rk_{M_G}-r_2)\leq -d_G-f_2,$$
	which is equivalent to 
\begin{align}\label{51}
d_1\leq \mu\left(M_G/S_2\right).
\end{align}
On the other hand, by Lemma \ref{Newsteadlem3} we have
$$
- d_1 = \lambda(L)\geq \lambda(G) = -\mu(M_G), 
$$
from which it follows that $d_1\leq \mu(M_G)$.
 Finally, by semistability of $M_G$  we obtain $d_1\leq \mu(M_G)\leq \mu\left(M_G/S_2\right)$, as desired.

Furtheremore, the equality $\mu(S)=\mu(M_E)$
can occur only if $G$ is strictly linearly semistable. According to \cite[Proposition 3.3]{MS},
this happens if and only if either $G\cong K_C(D)$ for some effective divisor $D$ of degree two, or $C$ is hyperelliptic and  $\deg(G)=2(h^0(G)-1)$. However, since $d_G\leq 2g-2$, the line bundle $G$ can not be of the form $K_C(D)$ for such a divisor. Therefore, the equality $\mu(S)=\mu(M_E)$ is impossible
 when $C$ is non hyperelliptic.

If $C$ is hyperelliptic and $\deg(G)=2(h^0(G)-1)$, then $E$ is an extension of the form $$0\rightarrow g^1_2\rightarrow E \rightarrow tg^1_2\rightarrow 0,$$
where $t:=h^0(G)-1$. 
 In this case, it follows that $E$ is strictly linearly semistable.
\end{proof}

\begin{corollary}
Let $E$ be a rank two bundle computing  
$\Cliff_2(C)$ and admitting a line subbundle of $L$ with $h^0(L)=2$. Then $M_E$ is semistable.

\end{corollary} \label{stability_pencils}
\begin{proof}
The assertion follows from Theorems \ref{indicecomputingthm1}, \ref{indicecomputingthm0}  and \ref{MSconjecturerank=2}.

\end{proof}

\begin{remark}
    In \cite[Theorem 5.10]{Ortetal} it was shown that in the moduli space $S_0(2,d, 4)$ 
   of $\alpha$ stable rank-2 coherent systems for small $\alpha$,  the locus parametrizing globally generated pairs $(E, V)$
    with the property that the kernel $M_{E,V}$ of the evaluation map  
    $V \otimes \mathcal{O}_C \rightarrow E $ is semistable lies
    (in some range of the degree) precisely away from the locus of pairs $(E, V)$ admitting subpencils, see Definition \ref{pencil}. 
    Although the $\alpha$-stability of 
    $(E,V)$ for small $\alpha$ implies semistability of $E$, this result does not contradict Corollary \ref{stability_pencils} because there is no necessarily a lift from the vector bundle $E$ computing $\Cliff_2(C)$ to the moduli space $S_0(2,d, 4)$.
\end{remark}

\subsubsection{$\rk E=2$, $h^0(E)\leq 6$}\ 
In this subsection, we prove Conjecture \ref{MSBconjecture}(ii) for rank two vector bundles admitting a small number of global sections.

\begin{proposition}\label{exactglobalsectionsLN}
	Let $E$ be a rank $2$ bundle computing $\Cliff_2(C)$ and $S$ is a maximal stable destabilizing subbundle of $M_E$. If $\rk F_S:=r+2\geq 3$, then every exact sequence
	\begin{align}\label{mainexactseq1}
		0\rightarrow \mathcal{O}_C^r\rightarrow F_S\rightarrow F\rightarrow 0,
	\end{align}
	induces an exact sequence on global sections.
\end{proposition}
\begin{proof}
	Since $H^0(F_S^*)=0$, $F$ must be non-trivial. Moreover, since  $F$ is globally generated, we have $h^0(F)\geq 3$. Additionally, $F$ does not admit a trivial quotient. Otherwise, $F_S^*$ would contain the trivial bundle as a subsheaf, which contradicts $H^0(F_S^*)=0$.
	
 Assume $h^0(F)=3$ and observe that $r+2\leq h^0(F_S) \leq r+h^0(F)=r+3$. If $h^0(F_S)=r+3$, then we have the assertion. If $h^0(F_S)=r+2$, 
     then $h^0(F_S)=\rk (F_S)$ which is possible only if $F_S$ is the trivial bundle, however, as we already mentioned, this is impossible.
     
Suppose $h^0(F)\geq 4$ and notice that
by Property \ref{Butler_prop} (4) of Butler's diagram,  $\deg F = \deg F_S < \deg E$, hence
    $$
    \mu(F) = \frac{\deg F}{2} \leq \mu(E) \leq g-1.
    $$
Since $H^0(F_S^*)=0$, the bundle $F_S$ admits no trivial quotient; therefore, $F$ can not admit a trivial quotient either.  
Then, we conclude
	\begin{align}\label{cliffordindicecomparisions}
		\frac{\deg(F)}{2}-h^0(F)+2\geq \Cliff(C)\geq \frac{d_E}{2}-h^0(E)+2,
	\end{align}
either by definition of $\Cliff_2(C)$ when $F$ is semistable, or 
by Lemma $\ref{LNlem}$ when $F$ is non semistable. 
Since $\deg F< \deg E$, inequality 
\eqref{cliffordindicecomparisions} implies that
 $h^0(F)<h^0(E)$. As in the proof of Theorem 
 \ref{specialcase}, it follows that if 
 the map $H^0(F_S)\rightarrow H^0(F)$ is not surjective, then 
$$\deg(F) < 2(h^0(F)-2)+2\Cliff_2(C).$$
	This contradicts (\ref{cliffordindicecomparisions}). 
\end{proof}
\begin{lemma}\label{bounddimF_S}
	Suppose 
	$E$ computes $\Cliff_2(C)$ and $S$ is a semistable subbundle destabilizing $M_E$ with maximal slope and $\rk F_S=r+2\geq 3$.  Then 
\begin{align}\label{sectionsinequality}
h^0(F_S)\leq h^0(E)+r-\frac{rd_E}{2(h^0(E)-2)}\cdot
\end{align}
\end{lemma}
\begin{proof}
	Since $S$ is of maximal slope and destabilizes $M_E$, and $\rk F_S=r+2\geq 3$, we have 
\begin{align}\label{inequality20}
\deg(F_S)\leq \frac{\dim W-\rk F_S}{\dim W-2}\cdot d_E,
\end{align}
 by (\ref{F_sdegreeinequality1}).
 Given a rank two quotient $F$ of $F_S$ 
 $$
 0\rightarrow \mathcal{O}^{\oplus r }\rightarrow F_S \rightarrow F \rightarrow 0,
 $$
$F$ satisfies the conditions of Lemma \ref{LNlem}, as in the proof of Proposition \ref{exactglobalsectionsLN}. Therefore,
	$$
    d_F-2h^0(F)\geq d_E-2h^0(E).
    $$
	Since, by Proposition \ref{exactglobalsectionsLN}, the sequence $0\to \oplus^r\mathcal{O}_C\to F_S\to F\to 0$ induces an exact sequence on sections, we obtain $\deg(F_S)-2(h^0(F_S)-r)\geq d_E-2h^0(E)$. Using (\ref{inequality20}), this gives 
	$$ 2[h^0(E)-h^0(F_S)+r]\geq \frac{d_E}{\dim W-2}\comma$$
as required.	
\end{proof}
\begin{remark}\label{remark10}
Observe that, by the Snake Lemma, we have an inclusion $ N_S\hookrightarrow Q_S,$ with $N_S$ as in (\ref{kernelalpha}), $Q_S:=M_E/S$. Furthermore, under the hypothesis of Lemma \ref{bounddimF_S}, we have $H^0(Q_S)=0$, by 
Lemma \ref{Abel}. So, $h^0(F_S)\leq h^0(I_S)\leq h^0(E)$. Inequality \ref{sectionsinequality} shows that except possibly for the case
$d_E=2(h^0(E)-2)$, it holds $h^0(F_S) <  h^0(E)$.   

\end{remark}
\begin{theorem}\label{h0_less_than 6}
	Suppose that $E$ computes the rank-two Clifford index with $h^0(E)\leq 6$, and that $C$ is non-hyperelliptic. Then $E$ is linearly (semi)stable if and only if $M_E$ is (semi)stable. 
\end{theorem}
\begin{proof}	
As in Theorem \ref{MSconjecturerank=2}, if $M_E$ is (semi)stable, then $E$ is linearly (semi)stable.	
 In order to prove the reverse statement, if $E$ is strictly linearly semistable, 
  with a globally generated invertible subsheaf $L\subset E$ satisfying $\lambda(L)=\lambda(E)$, then we have $\mu(M_L)=\mu(E)$; therefore $M_E$ would be strictly semistable. Hence, we may assume $E$ is linearly stable and prove that $M_E$ is stable. \\

   Take a subbundle $S\subset M_E$ of maximal slope. If either $\alpha_S$ is injective, or $\rk I=1$, then $\mu(S)\leq \mu(M_E)$, either by the linear stability of $E$ or by Lemma \ref{finalthrm}, respectively.  This particularly implies that we may assume $\rk(F_S)\geq 3$.
    
    If $\alpha_S$ is injective, then the equality  $\mu(S)=\mu(M_E)$ implies that $E$ is strictly linearly semistable. If $\rk I=1$, then we can apply Lemma \ref{finalthrm}, by which the equality $\mu(S)=\mu(M_E)$ implies that $\deg(E)=\delta_1\cdot(h^0(E)-2)$.  
     By Remark \ref{remark3}, if $h^0(E)=4$, then $\overline{I}$, the saturation of $I$, is a line subbundle of $E$ and satisfies $h^0(\overline{I})=2$, by Lemma \ref{Newsteadlem1}. Hence $\lambda(\overline{I})=\lambda(E)$, and so $E$ is strictly linearly semistable as desired. If 
$d_E=2(h^0(E)-2)$, then $E$ is strictly linearly semistable as in Theorem \ref{MSconjecturerank=2}. \\

In the course of our proof, we focus solely on the case $h^0(E)=6$, since the other cases, $h^0(E)\in \{4, 5\}$, are analogous and simpler.
Recall that $S$ can be assumed to be semistable, and we may also assume $d_E>2(h^0(E)-2)$. This, by Remark (\ref{remark10}) implies that $h^0(F_S)<h^0(E)$, so $\dim W\leq 5$. Since we are assuming $\rk F_S\geq 3$, we have $\dim W-\rk F_S\leq 2$.

 If $\dim W-\rk F_S=2$, then $S^*$ contributes to $\Cliff_2(C)$, and we have $\Cliff_2(S^*)\geq \Cliff_2(E)$. From this, if $\mu(S)\geq \mu(M_E)$, we obtain: 
 $$d_E-4h^0(S^*)\geq 2d_E-4h^0(E)=2d_E-24.$$
 This implies that $d_E\leq 8$, since $h^0(S^*)\geq 4$. By applying the Clifford theorem for $E$, together with \cite[Proposition 2]{Re}, we conclude that $E$ must be isomorphic to one of the following: $\mathcal{O}_C\oplus\mathcal{O}_C$, or $K\oplus K$, or $C$ is hyperelliptic and $E=g^1_2\oplus g^1_2$. However, all these possibilities are ruled out under the assumptions $\mu(E)\leq g-1$ and $d_E=8$. Hence, we conclude that $\mu(S)< \mu(M_E)$.
 
If $\dim W-\rk F_S=1$, then $S^*$ contributes to $\Cliff(C)$, and we have $\Cliff(S^*)\geq \Cliff_2(E)$. From this, the assumption $\mu(S)\geq \mu(M_E)$, would imply
	$$3d_E\leq 8,$$
	 which is clearly absurd. Hence, we conclude that $\mu(S)< \mu(M_E)$.
\end{proof}

\begin{corollary}\label{corsixsections}
Suppose $E$ computes $\Cliff_2(C)$,  $h^0(E)\leq 6$, and $C$ is non-hyperelliptic. Then $M_E$ is semistable.
\end{corollary}
\begin{proof}
It follows from Theorems \ref{h0_less_than 6}, \ref{indicecomputingthm1} and \ref{indicecomputingthm0}.
\end{proof}

\begin{remark}
If $S$ is as in Theorem \ref{bounddimF_S}, we have $h^0(F_S)\leq h^0(E)$, 
which follows from the properties of the Butler diagram. However, depending on the geometry of $C$, Theorem \ref{bounddimF_S} provides sharper inequalities.

Let $C$ be a general curve of genus $g=8$, and $E$ an extension of the form
$$0\rightarrow Q \rightarrow E \rightarrow K\otimes Q^* \rightarrow 0,$$
where $Q$ is a $g^1_3$. Then $E$ computes $\Cliff_2(C)$, and we have $h^0(E)=6$ and $\frac{d_E}{h^0(E)-2}=\frac{7}{2}$.

If $C$ is a general curve of odd genus $g\geq 9$, then the semistable vector bundles of type 
$$0\rightarrow Q \rightarrow E \rightarrow Q \rightarrow 0,$$
where $Q$ is a $g^1_{\delta_1}$ on $C$, compute $\Cliff_2(C)$. We have $h^0(E)=4$ and
$$\frac{d_E}{h^0(E)-2}=\Big[  \frac{g-1}{2} \Big]$$
See \cite[Proposition 7.2(3), Theorem 7.4(3)]{LNe}.
\end{remark}

\begin{remark}
(i) Since semistability is a crucial component in defining higher rank Clifford indices, the approach by Mistretta and Stoppino is not applicable to vector bundles of rank $\geq 2$.
One of our key results, building upon the work of Lange and Newstead, is that the Clifford index of globally generated co-rank zero subbundles of bundles
computing $\Cliff_2(C)$, is nearly comparable to $\Cliff(C)$. 


\noindent (ii) 
A pivotal component in Mistretta–Stoppino's approach to Conjecture (\ref{MSconjecture}) is a Castelnuovo-type lower bound on the degree of certain line bundles. Specifically, if the multiplication map $H^0(L)\otimes H^0(K)\rightarrow H^0(K\otimes L)$ fails to be surjective, then $\deg(L)\geq \delta_1\cdot (h^0(L)-1)$.
 However, this result does not readily extend to higher-rank vector bundles.
In contrast, our methodology connects the argument to the equivalence between linear stability and slope stability for line subbundles and quotient bundles of the involved rank two bundles. 
\end{remark}

\section{Examples:}

Lange and Newstead characterized the bundles that compute $\Cliff_2(C)$ and $\gamma_2(C)$ in several cases \cite{LNe} by
producing a list of such vector bundles. Recall that, according to the results of Section \ref{Linearstabilitysection}, all these bundles are linearly (semi)stable. Therefore,  the (semi)stablity of the Lazarsfeld–Mukai bundles of the corresponding  Lange–Newstead's bundles follows, in many cases, from the results of the previous section.

\begin{enumerate}
	\item  If 
	$E$ computes $\gamma_2(C)$ with $h^0(E)=3$, then its Lazarsfeld–Mukai bundle is a line bundle, which is stable. On the other hand, if $C$ computes $\gamma_2(C)$ with $h^0(E)\geq 4$, then it also computes $\Cliff_2(C)$. Thus, it suffices to discuss only the bundles that compute $\Cliff_2(C)$. It is straightforward to see that if $L_1$ and $L_2$ are line bundles with $\deg L_1=\deg L_2$ and $h^0(L_1)=h^0(L_2)$, then the Lazarsfeld–Mukai bundle of  $E:=L_1\oplus L_2$
	is semistable. Based on this observation, we will omit further discussion of semistability of bundles of this type.

\item  If $C$ is a curve of Clifford dimension $2$, then it is a smooth plane curve of degree $\delta_2$. If $H$ is the unique hyperplane bundle on $C$, then $h^0(H)=3$ and $H\oplus H$ is the only bundle computing $\Cliff_2(C)$. 

\item  If $C$ is a Petri curve of genus $5$, then the bundles computing $\Cliff_2(C)$ admit one of the following representations: \begin{itemize} 
	\item $0\rightarrow Q \to E\to Q \to 0$ with $h^0(E)=4$ 
	 \item  $0\to Q \to E\to K\otimes Q^* \to 0$ with $h^0(E)=4$, 
	 
	   \item or as $0\to M \to E\to K\otimes M^* \to 0$ with $\deg(M)=2$ and $h^0(M)=1$ and $E$ admits no subpencil.
	   \end{itemize}
	     The Lazarsfeld–Mukai bundles of all such bundles are semistable by Corollary \ref{corsixsections}, as they all satisfy $h^0(E)=4$.

\item The Lazarsfeld–Mukai bundle of any bundle computing $\Cliff_2(C)$ over a tetragonal curve of genus $6$ or $7$ is semistable by Corollary \ref{corsixsections}, as they all satisfy $h^0\leq 6$.

\item The Lazarsfeld–Mukai bundle of a bundle computing $\Cliff_2(C)$ over a general curve of genus $8$
 is semistable by Corollary \ref{corsixsections}, as they all satisfy $h^0\leq 6$. 

\item  For a general curve of genus $g\geq 7$ with $g\neq 8$, the following types of bundles may compute 
$\Cliff_2(C)$ 
 \begin{itemize}  
 \item Bundles which are extensions of the form $0\to Q_1\to E \to Q_2\to 0$ with $h^0(E)=4$ (including 
 the trivial extension). The Lazarsfeld--Mukai bundles of these bundles are semistable by Corollary \ref{corsixsections}, as they all satisfy $h^0(E)=4$. 
	
 \item Possibly, there are bundles which are extensions $0\to Q_1\to E \to K\otimes Q^*_2\to 0$ where all sections of  $K\otimes Q^*_2$ lift. The Lazarsfeld–Mukai bundles of such bundles, would be semistable by Theorem \ref{MSconjecturerank=2}. 
\item Stable bundles  that  do not possess a line subbundle with $h^0\geq 2$. The Lazarsfeld--Mukai bundles of such bundles, if they exist, are stable whenever $h^0(E)\leq 6$. 
\end{itemize}
\end{enumerate}
\section{Rank two Butler Conjecture}
In this section, we apply the concept of linear stability, together with the preceding arguments, to establish an affirmative result for the rank-two Butler conjecture within a specific range of degrees.
\subsection{Background on rank two Butler conjecture}

\begin{definition} \label{pencil}
	Let $C$ be a smooth curve of genus $g$, following the terminology in \cite{Ortetal} and \cite{CH2024},
	
	\begin{enumerate}
		\renewcommand{\labelenumi}{(\roman{enumi})}
		\item 
 we say that a coherent system $(E, V)$ on $C$ admits a subpencil, respectively a subnet, if there exists a rank one coherent subsystem $(L, W)$ such that $W\subset V\cap H^0(L)$ 
 with $\dim W = 2$, respectively $\dim W=3$. We denote
	by $P_0(n, d, k)$, respectively by $N_0(n, d, k)$, the locus in $S_0(n, d, k)$ of coherent systems $(E, V)$ admitting a subpencil, respectively a subnet;
		
		\item we denote by $T(n, d, n+m)$ the locus in $S_0(n, d, n+m)$ where the Butler conjecture is fulfilled, i.e. 
		$$ T(n, d, n+m):=\{(E, V)\in S_0(n, d, n+m) \mid (M^*_{E, V}, V^*)\in S_0(m, d, n+m) \}.$$ 

    \end{enumerate}
\end{definition}
We recall some facts required for the computations in this section.
According to \cite[Proposition 3.2]{Bradlowetal} the dimension of the space of extensions of coherent systems of the form 
\begin{equation}\label{cs_extension}
0\rightarrow (F_1, W_1)\rightarrow (E_0, V_0)\rightarrow (L_2, W_2)\rightarrow 0,
\end{equation}
where $(F_1, W_1)$ and $(L_2, W_2)$  are 
coherent systems of types $(n_1, d_1, k_1)$ and $(n_2, d_2, k_2)$ respectively, can be computed using the invariants $C_{21}$ and $C_{12}$ defined by
\begin{align}
    C_{21}:= n_1n_2(g-1)-d_1n_2 +d_2n_1 +k_2d_1-k_2n_1(g-1)-k_1k_2, 
\end{align}
and $C_{12}$ is defined by interchanging the indices in $C_{21}$. 

In the special case $n_1=n_2=1$, $k_1=2$ and $k_2=3$ these expressions reduce to
\begin{align}\label{C1221}
C_{21}=2d_1+d_2-2g-4 \quad , \quad C_{12} 
=d_2+d_1-g-5. \end{align}
On the other hand, the Zariski tangent space to the moduli space $G(n,d,k;\alpha)$ at a point $(E,V)\in G(n,d,k;\alpha)$ is isomorphic to $\mathrm{Ext}^1((E,V), (E,V))$ whose dimension is the Brill--Nother number $\beta(n,d,k)$ plus the dimension of $\mathrm{Ext}^2$. Then one can write  $\beta(n,d,k)$ in terms of the Brill--Noether number of the coherent systems appearing in the extension \eqref{cs_extension} as follows (\cite[Corollary 3.7]{Bradlowetal})
$$
\beta(n,d,k)=\beta(n_1,d_1,k_1) + \beta(n_2,d_2,k_2)  +C_{12} + C_{21} -1. 
$$

    We briefly review the main steps of the argument of \cite{Ortetal}. 
\begin{enumerate}
    \item $T(n, d, n+m)$ is open in $S_0(n, d, n+m)$. See \cite[Lemma 3.2]{Ortetal}. 
    
\item The map
 $D: (E, V) \mapsto (M^*_{E, V}, V^*) $ induces an isomorphism 
$$T(n, d, n+m)\cong T(m, d, n+m).$$
See \cite[Proposition 3.8]{Ortetal}. 

\item $P_0(n, d, n+m)$ is closed in $S_0(n, d, n+m)$ and under certain numerical hypothesis, $\dim P_0(2, d, 4)< \beta(2, d, 4)$.
We also have  $\dim S_0(n, d, n+m)\geq \beta(n, d, n+m)$, whenever $S_0(n, d, n+m)$ is non-empty. See \cite[Lemma 4.5]{Ortetal} and \cite[Proposition 5.8]{Ortetal}. 

\item Butler’s Conjecture holds for $(n,d,n+m)$ if and only if $T(n,d,n+m)$ is
 dense in $S_0(n,d,n+m)$ and $T(m,d,n+m)$ is dense in $S_0(m,d,n+m)$. See \cite[Theorem 3.9]{Ortetal}.

 \item For any irreducible component $X\subseteq S_0(2, d, 4)$ one has $\dim X\cap T(2, d, 4)=\dim X$.
\end{enumerate}
Summarizing points (1)–(3) above, we obtain that for any irreducible component 
$X$ of $S_0(2, d, 4)$, the locus $T(2, d, 4)\cap X$ is dense in $X$. Consequently, Butler's Conjecture would hold non-trivially once the non-emptiness of $T(2, d, 4)$ is established. This approach was employed in
 \cite{Ortetal} to prove Butler's conjecture for coherent systems of type $(2, d, 4)$.

Here, we adopt a similar approach to prove Butler’s Conjecture for coherent systems of type
$(2, d, 5)$ with 
\begin{align}\label{delta2}
2\delta_2\leq d\leq \frac{3g}{2}.\end{align}
In particular, we will show that $T(2,d,5)$ is
dense in $S_0(2,d,5)$ and $T(3,d,5)$ is dense in $S_0(3,d,5)$. We also establish the non-emptiness of $T(2,d,5)$ and $T(3,d,5)$.
\begin{remark}
Since $D(D(E, V))=(E, V)$, we obtain the equivalence
\begin{align}\label{equivalence}
T(n, d, n+m)\neq \emptyset  \iff T(m, d, n+m)\neq \emptyset.
\end{align}
However, the density of $T(n, d, n+m)$ in $S_0(n, d, n+m)$ does not imply the density of $T(m, d, n+m)$ in $S_0(m, d, n+m)$.
\end{remark}
\begin{convention}

Since, by the definition of $\delta_2$, the inequality \ref{delta2}, does not hold for all genera $g$, we impose the following restrictions on $g=g(C)$, throughout this section:
\begin{itemize}
    \item If $g\equiv 0 \pmod{3}$, then $g\geq 24,$

    \item If $g\equiv 1 \pmod{3}$, then $g\geq 28,$

    \item If $g\equiv 2 \pmod{3}$, then $g\geq 32.$
\end{itemize}
\end{convention}

\subsection{Density of $T(3, d, 5)\subseteq S_0(3, d, 5)$ }
\begin{theorem}\label{Butlerconjthrm3-1}   
Let $C$ be a general curve and 
 $(E, V)$ be a generated coherent system of type $(3,d, 5)$ with $2\delta_2\leq d\leq\frac{3g}{2}$. Assume furthermore that $(E, V)$ does not admit a subnet. Then, the following statements are equivalent.
 \begin{itemize}
 	\item $(E, V)$ is linearly (semi)stable,
 	
 	\item $M_{E, V}$ is slope (semi)stable.
 \end{itemize}
\end{theorem}
\begin{proof}	
If $M_{E,V}$ is (semi)stable, then $(E, V)$ is trivially linearly (semi)stable. Conversely, suppose that $(E, V)$
is linearly stable and consider an invertible sub-sheaf $S\subsetneq M_E$.

We consider the Diagram of Butler for $(E, V, S)$. If $\dim W=2$, where $W$ is as in Diagram \ref{Butlerdiagram}, then $F_S$ is a locally free sub-sheaf of $E$. The linear stability property of $E$ implies that $\mu(S)<\mu(M_{E,V})$.\\
If $\dim W=3$, then $\rk F_S=2$ and one can prove as in Lemma \ref{Butlerconjthrm3} that $\deg(F_S)\leq 2g-2$.
  Hence, we conclude by Lemma \ref{LNlem},
$$\Cliff_2(F_S)
\geq \Cliff(C)=\left[\frac{g-1}{2}\right],$$
whenever $h^0(F_S)\geq 4$. Therefore, $\deg(F_S)
\geq g+1, $
by which we obtain 
$$\mu(S)\leq -(g+1)<-\frac{3g}{4}\leq \mu(M_E).$$
If in the case $\dim W=3$, one had $h^0(F_S)=3$, then $W=H^0(F_S)$ and either $(F_S, H^0(F_S))$ 
is a rank two generated subsystem of $(E, V)$ or $I_S$ is an invertible subsheaf of $E$. In the former case we have $\mu(S)<\mu(M_E)$, by linear stability of $(E, V)$. While in the latter case we have $W\leq V\cap H^0(I_S)$, and so $(I_S, W)$ is a subnet of $(E, V)$, which contradicts the assumption. Summarizing, we have $\mu(S)<\mu(M_E)$ in the case $\dim W=3$. 

If $\dim W\geq 4$, then $\deg S^*\geq \delta_3= \lceil\frac{3g}{4}+3\rceil$.
Therefore, 
$$\mu(S)\leq -\left\lceil\frac{3g}{4}+3\right\rceil<-\frac{3g}{4}\leq \mu(M_E).$$		
\end{proof}
\begin{theorem}\label{Butlerconjthrm4-1}
Let $C$ be a general curve and 
 $(E, V)$ be a generated coherent system  of type $(3, d, 5)$ with $2\delta_2\leq d\leq\frac{3g}{2}$. If $E$ is semistable and $(E, V)$ does not admit generated rank two subsystems $(T, W)$ with $\dim W=3$, then, $(E, V)$ is linearly stable.
\end{theorem}
\begin{proof}
If $L$ is an invertible subsheaf of $E$, then $\deg(L)\leq \frac{d}{3}\leq \frac{g}{2}$, because $E$ is semistable. As $C$ is general $h^0(L)\leq 1$, so $E$ does not admit any non-trivial generated invertible subsheaf. 

Since $(E, V)$ admits no rank two subsystem $(T, W)$ with $\dim W=3$, so $E$ does not contain any rank two non-trivial generated subsheaf.

If $\rk F=3$, then $h^0(F)=4$ and $M^*_F$ is a line bundle with $h^0(M^*_F)\geq 4$. So $\deg(F)=\deg(M^*_F)\geq \delta_3$ and we obtain $\lambda(F)=\deg(F)\geq \delta_3>\frac{3g}{4}\geq \lambda(E, V)$. 
\end{proof}
We denote by $N^2_0(3, d. 5)$ the locus of generated coherent systems of type $(3, d, 5)$ admitting a generated subsystem of type $(2, d, 3)$.

\begin{corollary}\label{cor(3,d,5)}
$S_0(3, d, 5)\setminus N^2_0(3, d, 5)\subseteq T(3, d, 5).$
\end{corollary}
\begin{proof}
    If $(F, U)\in S_0(3, d, 5)\setminus N^2_0(3, d, 5)$, then $(F, U)$ is linearly stable by Theorem \ref{Butlerconjthrm4-1}.   
As $(F, U)$ is $\alpha$-stable for small values of $\alpha$, so $F$ is semistable, and any line bundle $L\subset F$ satisfies $\deg(L)\leq \frac{g}{2}< \delta_2$. Therefore, $(F, U)$ does not admit sub-nets. We conclude by Theorem \ref{Butlerconjthrm3-1}, that $(F, U)$ is stable. Therefore, $D(F, U)=(M^*_{F, U}, U^*)$ is $\alpha$-stable for all $\alpha$, in particular $(F, U)\in T(3, d, 5)$ by definition of $T(3, d, 5)$.
\end{proof}
\begin{proposition}\label{prop(3,d,5)}
Let $C$ be a general curve and $X\subseteq \overline{N^2_0(3, d, 5)}$ be an irreducible component. Then, either a general element of $X$ is linearly stable or we have 
$$\dim X <\beta(3, d, 5).$$
\end{proposition}
\begin{proof} 
If $(E, V)\in X$ is a general element, then $(E, V)$ sits in an exact sequence 
\begin{align}\label{secondboundseq}
\gamma: 0\rightarrow (F, W)\rightarrow (E, V)\rightarrow (L, \overline{W})\rightarrow 0,\end{align}
where $(F, W)$ and $(L, \overline{W})$ are generated coherent systems of types $(2, d_F, 3)$ and $(1, d_L, 2)$, respectively. 

Since $F$ is generated, we have
\begin{align}\label{Claim}d_F\ge \delta_2.\end{align}
Indeed,
since $E$ is semistable, we have 
$d_F\leq g,$
which implies that $\det(F)\neq 2g^1_{\delta_1}$. Therefore, we may apply \cite[Proposition 2.2]{T07}, from which it follows that in the sequence $0\rightarrow \mathcal{O}_C \rightarrow F \rightarrow \det(F) \rightarrow 0$, not all sections of $\det(F)$ lift to $F$. In particular, we have $h^0(\det(F))\geq 3$, and hence 
$\deg(\det(F))\geq \delta_2$. This establishes \ref{Claim}.

Suppose first that every non-trivial and generated rank two locally free sub-sheaf $F\subset E$ 
satisfies $h^0(F)\geq 4$.
Then, $(E,V)$ is linearly stable. In order to see this, since $E$ is semistable and $\mu(E)<\delta_1$, $E$ does not contain any non-trivial invertible subsheaf. If $F\subset E$ is any rank two globally generated subsheaf, then $F$ does not admit a  trivial quotient bundle, otherwise $F$ would have a subpencil, which is not 
possible since $\mu(E)<\delta_1$.   According to Lemma \ref{LNlem}  we have 
$d_F\geq g+2$ . So, 
$$\lambda(F, W)\geq  g+2  >\frac{3g}{4}\geq \lambda(E, V),$$
as required.

Assume secondly, that $(E, V)$ contains a subsystem $(F, W)$
as in \ref{secondboundseq} and $W=H^0(F)$. We consider the sequence 
 $0\rightarrow M^*_{L, \overline{W}}\rightarrow M^*_{E,V}\rightarrow M^*_F\rightarrow 0,$
 which by $M_{L, \bar{W}}=L^*$, is actually the sequence 
\begin{align}\label{exactseq32}
\eta: 0\rightarrow L\rightarrow M^*_{E,V}\rightarrow M^*_F\rightarrow 0.
\end{align}
Notice that any element of $H^0(F)^*\subseteq H^0(M^*_F)$ lifts to a section of $M^*_{E, V}$, because the elements of $H^0(F)^*=W^*$ lift to $V^*$. Therefore, $\eta\in \Coker(P_\eta)$, with $P_\eta$ the multiplication map associated to the extension $\eta$ in \ref{exactseq32}.
We distinguish two cases: 
\begin{itemize}
    \item For any $F$ appearing in the sequence $\gamma$ with $h^0(F)=3$, we have $h^0(M^*_F)\geq 4$,
     \item There exists $F$ appearing in the sequence $\gamma$ with $h^0(F)=3$ such that $H^0(F)^*= H^0(M^*_F)$.
\end{itemize}
If we are in the first case, then $\deg(M^*_F)\geq \delta_3$. So $\lambda(F)\geq \delta_3> \frac{3g}{4} \geq \lambda(E, V)$ and the coherent system $(E, V)$ is linearly stable, as previously. 

Denoting by $X_0$ the locus of coherent systems $(E, V)$ satisfying the property for the second case, we set 
$$t:=\min\{h^0(L): (L, \overline{W})\quad \!\!\!\!\mbox{appears in the sequence \ref{secondboundseq} with}\quad \!\!\!\! (E,V)\in X_0\}.$$
Then $t\geq 2$. We shall prove 
\begin{align}\label{theineq3}
\dim X_0< \beta(3, d, 5)=\beta(2, d, 5).\end{align}
As \ref{exactseq32} is obtained from \ref{secondboundseq} uniquely, we have 
\begin{align}\label{111}
    \dim X_0\leq \beta(1, d_F, 3)+\beta(1, d_L, t)+\dim \Coker(P_\eta)-1.
\end{align}
\textit{First subcase:} $t\geq 3$.  
Take a coherent system $(E,V)\in X_0$ such that $(E, V)$ admits a sequence as \ref{secondboundseq} with $h^0(L)=t$.
For general $q\in C$, the multiplication map 
$$P^q_\eta: H^0(M^*_F(-q))\otimes H^0(K\otimes L^*) \rightarrow  H^0(K\otimes M^*_F\otimes L^*(-q)),$$
satisfies 
$$\dim \ker P^q_\eta=h^0(K\otimes L^*\otimes M_F(q))\leq 1.$$ 
Hence $\dim \im P_{\eta}^q \geq 2h^0(K\otimes L^*) - 1$, and
 since $\dim \im P_\eta \geq  \dim \im P_{\eta}^q \geq 2h^0(K\otimes L^*) - 1$, we have 
\begin{align}\label{Coker(P)}
\dim \Coker P_\eta \leq  h^0(M_F^*\otimes K \otimes L^*) -2h^0(K\otimes L^*) +1.
\end{align}

Assume first $t=3$. Then, $(M_{E,V}^*, V^*) \in N_0(2,d,6)$ and we have 
$$\beta(2,d, 5)=(\beta(2,d,5)-\beta(2,d,6))+\beta(2,d,6)=(2g-d+9)+[\beta(1,d_F,3)+\beta(1,d_L,3)+\bar{C}_{21}+\bar{C}_{12}-1],$$
with $\bar{C}_{21}$ and $\bar{C}_{12}$ the invariants associated to the types $(1, d_F, 3)$ and $(1,d_L, 3)$. 
From this, by \ref{111}, the inequality \ref{theineq3} holds if
$$\dim \Coker(P_\eta)< \bar{C}_{21}+\bar{C}_{12}+(2g-d+9).$$
On the other hand, a direct computation implies 
$$\bar{C}_{21}+\bar{C}_{12}=3d-4(g-1)-18.$$
Therefore, by \ref{Coker(P)}, we have to prove 
\begin{align}\label{-11}
     h^0(M_F^*\otimes K \otimes L^*)< 2h^0(K\otimes L^*)+2d-2g-6.\end{align}
By Riemann--Roch, $h^0(K\otimes L^*)=g-d_L+2$, therefore the inequality \ref{-11} turns to be
$$ h^0(M_F^*\otimes K \otimes L^*) <2(g-d_L+2)+2d-2g-6=2d-2d_L-2.$$
This, by a simple computation, is equivalent to 
\begin{align}\label{-12}
 h^0(L\otimes M_F)<d_F+d_L-g-1=d-g-1.\end{align}
Since, by \ref{Claim}, we have $\deg(L\otimes M_F)\leq \frac{g}{6}$, so 
\begin{align}\label{Claimapplication}
h^0(L\otimes M_F)\leq 1.\end{align}
Now, \ref{-12} holds as follows  
$$d-g-1\geq 2\delta_2-g-1\geq 2\left(\frac{2g}{3}+1\right)-g-1=\frac{g}{3}+1>1\geq h^0(L\otimes M_F),$$
for $g\geq 4$.

Still within the first subcase, suppose that
$t\geq 4$. According to \cite[Corollary 3.7]{Bradlowetal},  
$$\beta(2, d, 5)=\beta(1, d_F, 3)+\beta(1, d_L, 2)+C_{21}+C_{12}-1,$$ with $C_{21}$ and $C_{12}$ the invariants associated to the types $(1, d_F, 3)$ and $(1, d_L, 2)$. So, \ref{theineq3} will hold if 
$$ \dim \Coker(P_\eta)< [\beta(1, d_L, 2)-\beta(1, d_L, t)]+C_{21}+C_{12}.$$
Taking into account the equality $$C_{21}+C_{12}=-3(g-1) +3d_L +2d_F -12,$$
obtained from \ref{C1221}, a straightforward calculation shows that
$$[\beta(1, d_L, 2)-\beta(1, d_L, t)]+C_{21}+C_{12}=t h^0(K\otimes L^*) -5g +5d_L +2d_F -11.$$
Therefore, it remains to prove
\begin{align}\label{cokerdim}
\dim \Coker(P_\eta)< t\cdot h^0(K\otimes L^*) -5g +5d_L +2d_F -11.
\end{align}
Hence, in view of \ref{Coker(P)}, it suffices to show that 
\begin{align}\label{-5}
h^0(M_F^*\otimes K \otimes L^*) < (t+2)h^0( K \otimes L^*) -5g+5d_L +2d_F-12.
\end{align}
 By \ref{Claimapplication},
 we have
\begin{align}\label{-6}
h^0(M_F^*\otimes K \otimes L^*) = h^0(M_F\otimes L)+(g-1)-d_L+d_F \leq g+d_F-d_L.
\end{align}
On the other hand, we have  
$h^0( K \otimes L^*)=t+g-d_L-1.$ Hence  
by \eqref{-6}, inequality \eqref{-5} will follow once we show that
$$
g+d_F-d_L< (t+2)(t+g-1-d_L) - 5g +5d_L +2d_F -12, 
$$
which is equivalent to 
$$
0< (t+2)(t+g-1-d_L) - 6g +6d_L +d_F -12=:B.
$$
Since $d_L\leq g$, the invariant $B$ satisfies
$$B=(t-4)g-(t-4)d_L+(t+2)(t-1)+d_F-12\geq d_F+18-12>0,$$
as needed. 

\textit{Second subcase:} If $t=2$, i.e., for some rank one coherent system $(L, \overline{W})$ appearing in the sequence \ref{secondboundseq}, we have $\overline{W}=H^0(L)$, then $V=H^0(E)$ and the sequence \ref{secondboundseq} is a sequence of complete coherent systems. Furthermore, all sections of $H^0(L)$ lift to $E$. Therefore, the extension $\eta$ in \ref{secondboundseq} belongs to $\Coker P_\gamma$, where $P_\gamma$ is the multiplication map associated with this sequence. 
If $\mathcal{F}$ denotes the locus over which $F$ varies, then 
$$\dim X_0\leq \dim \mathcal{F}+\beta(1,d_L,2)+\dim \Ext^1(L, F)-1-\dim \im (P_\gamma).$$
The association $F\mapsto D(F, H^0(F))\in S_0(1, d_F, 3)$ is injective, implying that 
$$\dim X_0\leq \beta(1, d_F, 3)+\beta(1,d_L,2)+\dim \Ext^1(L, F)-1-\dim \im (P_\gamma).$$
Consider that for any $E\in \Ext^1(L, F)$ we have $M_E^*\in \Coker(P_\eta)$, therefore the injective association $(E,H^0(E))\mapsto D(E, H^0(E))$ implies 
\begin{align}
\dim X_0\leq \beta(1,d_F,3)+\beta(1,d_L,2)+\dim \Coker(P_\eta)-1-\dim \im (P_\gamma).
\end{align}
Therefore, \ref{theineq3} will hold if  
\begin{align}\label{-13}
    \dim \Coker(P_\eta)-\dim \im (P_\gamma)<C_{21}+C_{12}=-3(g-1) +3d_L +2d_F -12.\end{align}
By \ref{Coker(P)} and since $h^0(K\otimes L^*)=g-d_L+1$ we have 
$$\dim \Coker P_\eta\leq  h^0(M_F^*\otimes K \otimes L^*) -2h^0(K\otimes L^*)+1=h^0(M_F^*\otimes K \otimes L^*) -2\cdot(g-d_L+1)+1.$$
Since $h^0(L\otimes M_F)\leq 1$ by \ref{Claimapplication}, we obtain
$$h^0(K\otimes L^*\otimes M_F^*)=h^0(L\otimes M_F)-d_L+d_F+g-1\leq g+d_F-d_L.$$
Thus
$$\dim \Coker P_\eta\leq g+d_F-d_L-2\cdot(g-d_L+1)+1=
d_F+d_L-g-1=d-g-1.$$
On the other hand,
$\dim \im(P_\gamma)=2.(2g-d_F+1)-h^0(K\otimes L^*\otimes F^*),$ and since $F$ is generated,  
$$h^0(L\otimes F)\leq h^0(L)+h^0(L\otimes \det(F))\leq 2+d-g+3\leq \frac{g}{2}+5.$$ 
Hence,
$$h^0(K\otimes L^*\otimes F^*)=h^0(L\otimes F)+2(g-1)-d_F-2d_L\leq 
\frac{5g}{2}+3-d_F-2d_L.$$
Therefore,
$$\dim \im(P_\gamma)\geq (4g-2d_F+2)-\left(\frac{5g}{2}+3-d_F-2d_L\right)=\frac{3g}{2}-d_F+2d_L-1.$$
Summarizing, we conclude
$$\dim \Coker(P_\eta)-\dim \im (P_\gamma)\leq (d-g-1)-\left(\frac{3g}{2}-d_F+2d_L-1\right)=d+d_F-2d_L-\frac{5g}{2}.$$
Hence, inequality \ref{-13} will hold if 
$$d+d_F-2d_L-\frac{5g}{2}<-3(g-1)+3d_L+2d_F-12,$$
which follows immediately from $d_L>\delta_1$.
\end{proof}

\begin{corollary}
If non-empty, then $T_0(3,d,5)\subseteq S_0(3, d, 5)$ is dense. 
\end{corollary}
\begin{proof}
This follows from Theorem \ref{Butlerconjthrm3-1}, Theorem \ref{Butlerconjthrm4-1} and Proposition \ref{prop(3,d,5)}.  
\end{proof}

\subsection{Density of $T(2, d, 5)\subseteq S_0(2, d, 5)$}

\begin{theorem}\label{Butlerconjthrm3}
Let $C$ be a general curve and 
 $(E, V)$ is a generated coherent system of type $(2,d, 5)$ with $2\delta_2\leq d\leq\frac{3g}{2}$. Then, the following statements are equivalent.
 \begin{itemize}
 	\item $(E, V)$ is linearly (semi)stable,
 	
 	\item $M_{E, V}$ is slope (semi)stable.
 \end{itemize}
\end{theorem}
\begin{proof}	
If $M_{E,V}$ is (semi)stable, then $(E, V)$ is trivially linearly (semi)stable. Conversely, suppose $(E, V)$
is linearly stable, and consider a proper semistable subbundle $S\subsetneq M_E$. 
If $\rk S=1$ and $S$ destabilizes $M_E$ then 
$\deg(S^*)\leq\frac{d_E}{3}\le \frac{g}{2}$.  This is impossible on a general curve because $h^0(S^*)\geq 2$.

Now assume $\rk S=2$, so $3\leq \dim W\leq 5$. Note that $Q=M_E/S$ is a subsheaf of the trivial bundle 
$V/W\otimes \mathcal{O}_C$. Therefore, $\deg Q\leq 0$, implying $\deg(M_E)\leq \deg S$, and consequently 
$$\deg S^*\leq \deg M^*_E =\deg E\leq \frac{3g}{2}\leq 2g-2.$$
\noindent Hence, if $4\leq \dim W\leq 5$, then $S^*$ contributes in $\Cliff_2(C)$.
The inequality $\mu(S)\geq \mu(M_E)$ would imply 
$$\Cliff_2(C)\leq \frac{-\deg(S)}{2}-2\leq \frac{d_E}{3}-2\leq \frac{g}{2}-2,$$
which contradicts the equality $\Cliff_2(C) = \Cliff(C)$. This, by \cite{BF}, is absurd on a general curve.

If $\dim W=3$, then $\alpha_S$ is injective and the result follows from the definition of linear (semi)stability for $E$. 		
\end{proof}

\begin{theorem}\label{Butlerconjthrm4}
Let $C$ be a general curve and 
 $(E, V)$   
 a generated coherent system of type $(2, d, 5)$ with $2\delta_2\leq d\leq\frac{3g}{2}$. If $(E, V)$ does not  admit a rank one subsystem $(L, W)$ with $\dim W=3$, i.e.,  $(E, V)\notin N_0(2, d, 5)$, then $(E,V)$ is linearly stable.
\end{theorem}
\begin{proof}
If $(L, W)$ is a globally generated subpencil of $(E, V)$, and $\dim W=2$, then $\lambda(L, W)=d_L\geq \delta_1> \frac{g}{2}\geq \frac{d_E}{3}= \lambda(E, V)$. 
	
Assume $F$ is a  globally generated and rank two locally free subsheaf of $E$. since $\lambda(F, W)\geq \lambda(F)$, to prove $\lambda(F, W)\geq \lambda(E, V)$ it suffices to consider the case $W=H^0(F)$. Thus, we assume $h^0(F)=4$. 
Note that $F$ can not be an extension of the form $0\rightarrow L\rightarrow F\rightarrow \mathcal{O}_C \rightarrow 0$, because otherwise ($L, H^0(L)$) would be a rank one subsystem of $(E, V)$ with $h^0(L)\geq 3$.

\noindent Suppose $F$ is semistable. Since, as in Theorem \ref{Butlerconjthrm3}, we have $\mu(F)\leq 2g-2$. Hence $F$ contributes to $\Cliff_2(C)$, and by Mercat's Theorem for rank two bundles, \cite{BF}, we have
\begin{align}\label{cliffF1}
    \Cliff(F)\geq \Cliff_2(C)=\Cliff(C)\geq \frac{g-1}{2}-1.
    \end{align}
If $F$ is non-semistable and admits no trivial quotient, then the inequality (\ref{cliffF1}) holds by Lemma \ref{LNlem}.
In either cases we have $g+1\leq\deg(F)$.
Therefore, for each subspace $U \subseteq H^0(F)$ generating $F$ with $U\subset V\cap H^0(F)$, we have 
$$\lambda(E, V)\leq\frac{g}{2}<\frac{g+1}{2}\leq\frac{\deg(F)}{2}=\lambda(F) \leq\lambda(F, U).$$

	Assume finally that
	 $h^0(F)=3$, then we  have
	$\lambda(F)=d_F\geq \delta_1> \frac{g}{2}\geq \lambda(E, V)$, as required.	
\end{proof}

\begin{corollary}\label{Butlerconjthrm2}
	Suppose $C$ is a general curve and $d$ is an integer with $2\delta_2\leq d\leq\frac{3g}{2}$. Then, 
 we have 
	$$S_0(2, d, 5)\setminus N_0(2, d, 5)\subseteq T(2, d, 5).$$
\end{corollary}
\begin{proof}
   This is a direct consequence of Theorem \ref{Butlerconjthrm3} and Theorem \ref{Butlerconjthrm4}.
\end{proof}

\begin{lemma}\label{closedness}
The locus $N_0(n, d,n+m)$ is closed in $S_0(n, d, n+m)$ and $N_0(2, d,5)$ does not fill any irreducible component of $S_0(2, d, 5)$.
\end{lemma}

\begin{proof}
A verbatim repetition of the proof of \cite[Lemma 4.5]{Ortetal} shows that $N_0(n, d, n+m)$ is closed in $S_0(n, d, n+m)$.

In order to prove the second statement, likewise in Proposition \ref{prop(3,d,5)}, we  shall prove 
\begin{align}\label{theineq3-1}
\dim X< \beta(2, d, 5),\end{align}
for any irreducible component $X\subset N_0(2, d, 5)$. If $(E, V)\in X$ is a general element, then $(E, V)$ sits in an exact sequence 
\begin{align}\label{secondboundseq2}
\gamma_2: 0\rightarrow (F, W)\rightarrow (E, V)\rightarrow (L, \overline{W})\rightarrow 0,\end{align}
where $(F, W)$ and $(L, \overline{W})$ are generated coherent systems of types $(1, d_F, 3)$ and $(1, d_L, 2)$, respectively. 

Since $h^0(F)\geq 3$, we have $d_F\ge \delta_2$ and as $E$ is semistable, we have $d_F\leq \frac{3g}{4}<\delta_3$, so $h^0(F)=3$. Then any subsystem $(F, W)$ appearing in the sequence \ref{secondboundseq2} is complete. Additionally, since $F\neq 2g^1_{\delta_1}$, we have $h^0(M_F^*)=h^0(F)=3$ by \cite[Theorem 2.4]{T07}. 

As in Proposition \ref{prop(3,d,5)}, we consider the sequence 
\begin{align}\label{exactseq32-1}
\eta_2: 0\rightarrow L\rightarrow M^*_{E,V}\rightarrow M^*_F\rightarrow 0,
\end{align}
and we have $\eta_2\in \Coker(P_{\eta_2})$, with $P_{\eta_2}$ the multiplication map associated to the extension $\eta_2$ in \ref{exactseq32-1}.
We also set 
$$t:=\min\{h^0(L): (L, \overline{W})\quad \!\!\!\!\mbox{appears in the sequence \ref{secondboundseq2} with}\quad \!\!\!\! (E,V)\in X\}.$$
Again, we have 
\begin{align}\label{111-2}
    \dim X\leq \beta(1, d_F, 3)+\beta(1, d_L, t)+\dim \Coker(P_{\eta_2})-1.
\end{align}
\textit{First subcase:} If $t=2$, i.e., for some rank one coherent system $(L, \overline{W})$ appearing in the sequence \ref{secondboundseq2}, we have $\overline{W}=H^0(L)$, then $V=H^0(E)$ and the sequence \ref{secondboundseq2} is a sequence of complete coherent systems. Furthermore, all sections of $H^0(L)$ lift to $E$. Therefore, the extension $\gamma_2$ in \ref{secondboundseq2} belongs to $\Coker P_{\gamma_2}$, where $P_{\gamma_2}$ is the multiplication map associated with this sequence and we have
\begin{align}
\dim X\leq \beta(1,d_F,3)+\beta(1,d_L,2)+\dim \Coker (P_{\gamma_2})-1.
\end{align}
Therefore, \ref{theineq3-1} will hold if  
\begin{align}\label{-131}
\dim \Coker (P_{\gamma_2})<C_{21}+C_{12},\end{align}
which in this case $C_{21}+C_{12}=3d_L+2d_F-3g-9$.
Since $\deg (K\otimes F^*\otimes L^*) < \delta_2$,
$$\dim \Ker (P_{\gamma_2})=h^0(K\otimes L^*\otimes F^*)\leq 2.$$
Hence,
$$\dim \Coker P_{\gamma_2}\leq  h^0(K\otimes L \otimes F^*) -2h^0(K\otimes F^*)+2=h^0(K\otimes L \otimes F^*)-2\cdot(g-d_F+2)+2.$$
Since $\deg(F\otimes L^*)\leq 0$ and $L\neq F$, so $h^0(F\otimes L^*)=0$. 
 Thus,
$$\dim \Coker P_{\gamma_2}\leq (g+d_L-d_F-1)-2\cdot(g-d_F+2)+2=
d_F+d_L-g-4=d-g-3.$$
So we have to prove  $d-g-3< 3d_L+2d_F-3g-9$
which is equivalent to 
\begin{align}\label{-7}
2g+6< d_L+d,\end{align}
If $d=2\delta_2$, then it is shown in the proof of \cite[Theorem D]{CH2024} that a coherent system $(E, V)\in S_0(2,d, 5)$ does not admit any subnet. Hence $N_0(2, 2\delta_2, 5)=\varnothing$, therefore we assume $d\geq 2\delta_2+1$. Now, the inequality \ref{-7} follows from  $d+d_L\geq 3\delta_2+1 >2g+6 $. 

\textit{Second subcase:} $t\geq 3$.  
Take a coherent system $(E,V)\in X$, admitting a sequence as \ref{secondboundseq2} with $h^0(L)=t$. Consider the multiplication map
$$P_{\eta_2}: H^0(M^*_F)\otimes H^0(K\otimes L^*) \rightarrow  H^0(K\otimes M^*_F\otimes L^*),$$
associated to the exact sequence \eqref{exactseq32-1}.
Since $\det(M_F^*)=F$, we have the following exact sequence 
\begin{align}\label{-4}
0\rightarrow \mathcal{O}_C\rightarrow M_F^* \stackrel{\theta}{\rightarrow} F\rightarrow 0.    
\end{align}
 Likewise in \cite{Baj2020},  we derive the following diagram 
{\small
	\begin{align}\label{N0(2,d,5)diagram}
		\xymatrix{
			0\ar[r]&H^0(K\otimes L^*) \ar[d]^{=} \ar[r],&H^0(K\otimes L^*)\otimes H^0(M_F^*)\ar[d]^{P_{\eta_2}} \ar[r] &H^0(K\otimes L^*)\otimes H  \ar[d]^{P}\ar[r]&0\\
			0\ar[r] &H^0(K\otimes L^*)\ar[r]^{f_1}& H^0(K\otimes L^*\otimes M_F^*)\ar[r]^{f_2}& H^0(K\otimes F\otimes L^*),&
		}
	\end{align}
}
in which $H:=\im(H^0(\theta))$. See also \cite{ABH}.
The Snake Lemma applied to the Diagram \ref{N0(2,d,5)diagram} gives 
$$\dim \Ker(P_{\eta_2})\leq \dim \Ker(P).$$
Consider now that $\dim H=2$ and $F$ is generated by $H$. So, by base point free pencil trick, we obtain 
$$\Ker(P)=H^0(K\otimes L^*\otimes F^*).$$
Therefore, $\dim \Ker(P)\leq 2$, so $\dim \Ker(P_{\eta_2})\leq 2$
 and 
$\dim \im (P_{\eta_2})\geq 3\cdot(h^0(L)+g-d_L -1)-2.$ 
Hence,
\begin{align}\label{-3}
\dim \Coker P_{\eta_2} \leq  
h^0(K\otimes M_F^* \otimes L^*) -3\cdot (g-d_L +2)+2.
\end{align}

Now, by \ref{111-2}, we shall  prove 
$$ 
\dim \Coker(P_{\eta_2})<[\beta(1,d_L,2)-\beta(1,d_L,t)]+(3d_L+2d_F-3g-9),$$
for $t\geq 3$, and by \ref{-3}, this will hold if 
$$ h^0(K\otimes M_F^* \otimes L^*) -3\cdot (g-d_L +2)+2< [\beta(1,d_L,2)-\beta(1,d_L,t)]+(3d_L+2d_F-3g-9)
$$
which is equivalent to 
\begin{align}\label{-5-1}
  h^0(K\otimes M_F^* \otimes L^*)< [\beta(1,d_L,2)-\beta(1,d_L,t)]+2d_F -5=(t-2)(g-d_L+1)+t(t-2)+2d_F -5.
  \end{align}
Now, in the exact sequence 
$$0\rightarrow K\otimes L^*\rightarrow K\otimes L^*\otimes M_F^* \rightarrow  K\otimes L^*\otimes F\rightarrow 0,$$
obtained from \ref{-4}, we have $h^0(K\otimes L^*\otimes F)\leq g$,  because $\deg(F\otimes L^*)\leq 0$.
Therefore, 
\begin{align}\label{-6-1}
h^0(K\otimes L^*\otimes M_F^*)\leq g+(t+g-d_L -1).
\end{align}
Summarizing, by \ref{-5-1} and \ref{-6-1} we have to prove  
$$g<(t-3)(g-d_L+1)+(t-1)(t-2)+2d_F-5.$$ 
Since $t\geq 3$ and $g-d_L+1\geq 0$, it suffices to prove $$(t-1)(t-2)+2d_F-5>g,$$ which by $d_F\geq \delta_2$ is immediate.

\end{proof}

\begin{corollary}
If non-empty, then $T_0(2,d,5)\subseteq S_0(2, d, 5)$ is dense. 
\end{corollary}
\begin{proof}
This follows from Theorem \ref{Butlerconjthrm4}, Theorem  \ref{Butlerconjthrm3} and Lemma \ref{closedness}.  
\end{proof}

\subsection{Non-Emptiness and Butler's Conjecture}

We recall the following Lemma from \cite{CH2024}, which is a key tool in producing rank two vector bundles with prescribed number of sections.
\begin{lemma}[{\normalfont\cite[Lemma 4.1]{CH2024}}]\label{GeorgeAbel}
 Let $L_1$ and $L_2$ be generated line bundles over $C$. For $i\in \{1,2\}$, write
 $\deg L_i =: l_i$ and $h^0(L_i) =: k_i$. Suppose that
\begin{align}\label{GeorgeAbel1} 
 l_2 > k_1k_2 +(k_2-1)(g-1-l_1).\end{align}
 Then, there exists a nontrivial extension $0 \rightarrow L_1 \rightarrow E \rightarrow L_2 \rightarrow 0$ in which all sections of
 $L_2$ lift to $E$. In particular, the coherent system $(E,H^0(E))$ is of type $(2,l_1 + l_2, k_1 + k_2)$ and generated.
\end{lemma}

\begin{lemma}\label{nonemptiness}
\label{Butlerconjthrm5}
(i)	The locus $S_0(2, d, 5)$ is nonempty for $d\geq 2\delta_2$.

\noindent (ii)  The locus $T(2, d, 5)$ is nonempty whenever $ 2\delta_2\leq d \leq \frac{3g}{2}$.
\end{lemma}
\begin{proof}
(i) We apply an argument analogous to that used in the proof of \cite[Theorem D]{CH2024}.

\noindent Let $\beta\geq 0$ be an integer. 
Since $C$ is a general, there are globally generated line bundles 
$L_1\in W^1_{\delta_2+\beta}$ and $L_2\in W^2_{\delta_2+\beta+\epsilon}$ with $\epsilon\in \{0, 1\}$. 

 As the locus $S_0(2, d, 5)$ is proved to be non-empty for $d=2\delta_2$ in \cite[Theorem D]{CH2024}, we prove non-emptiness of $S_0(2, d, 5)$ for $d\geq 2\delta_2+1$.
If $\beta=0$, then we take $\epsilon=1$ and observe that
$$\delta_2+1 > 6+2\cdot (g-1-\delta_2).$$
Therefore, by Lemma \ref{GeorgeAbel},  
there exists a nontrivial extension $e: 0\rightarrow L_1\to E\to L_2\to 0$, with $L_1\in W^1_{\delta_2}$ and $L_2\in W^2_{\delta_2+1}$ such that all sections of $L_2$ lift to $E$.

If $\beta\geq 1$, then the inequality 
$$\delta_2+\beta+\epsilon > 6+2\cdot (g-1-\delta_2-\beta),$$
holds for all $g$.  
 So, again by Lemma \ref{GeorgeAbel}, 
there exists a nontrivial extension $e: 0\rightarrow L_1\to E\to L_2\to 0$ 
such that all sections of $L_2$ lift to $E$. Since $L_1$ and $L_2$ are globally generated, so is the coherent system $(E, H^0(E))$. 
Observe that for an arbitrary invertible subsheaf $M$ of $E$, either $L_1\cap M=0$ or $L_1\cap M$ is a subsheaf of $L_1$ and $M$ is actually a subsheaf of $L_1$, so $h^0(M)\leq 2$. In the former case there is a 
non-zero map $M \rightarrow L_2$.
Since the extension $(e)$ is non-trivial, 
$\deg M \leq \deg L_2-1 = \delta_2 + \beta +\epsilon - 1 $ (otherwise $M\simeq L_2$). 
Hence $\deg M \leq  \delta_2 + \beta$. 
Since $L_2$ is globally generated, its non-trivial subsheaves have at most two sections.  So $h^0(M) \leq 2 $, for any  non-trivial subsheaf $M$ of $E$. Therefore, 
$$\mu_\alpha(M, W)\leq \delta_2+\beta+2\alpha< \delta_2+\beta+\frac{5\alpha}{2}=\mu_\alpha(E, H^0(E)),$$
for $W\leq H^0(M)$, implying that $(E, H^0(E))$ is $\alpha$-stable for any $\alpha >0$.  
Consequently $S_0(2, d, 5)$ would be non-empty for $d$ in the given range.

(ii) The bundles $E$ constructed in part (i) do not admit any invertible sub-sheaf with at least $3$ sections. Hence, they are linearly stable  by Lemma \ref{Butlerconjthrm4}. Therefore, $M^*_E$ is stable by Lemma \ref{Butlerconjthrm3}, so $M^*_E$ is $\alpha$-stable for all $\alpha>0$. 
We conclude $(E, H^0(E))\in T(2, d, 5)$, by definition. 
\end{proof}

\begin{lemma}\label{nonemptiness2}
 The locus $T(3, d, 5)$ is non-empty for $ 2\delta_2\leq d\leq \frac{3g}{2}$. 
\end{lemma}
\begin{proof}
 This is immediate by \ref{equivalence} and Lemma \ref{nonemptiness}(ii).
\end{proof}
\begin{theorem}\label{Butlerfor2d5}
Suppose $C$ is a general curve and $2\delta_2\leq d\leq \frac{3g}{2}$. Then, the Butler conjecture holds non-trivially for coherent systems of type $(2, d, 5)$.
\end{theorem}
\begin{proof}
It follows from Lemma \ref{closedness}, Corollary \ref{Butlerconjthrm2} and Lemma \ref{nonemptiness} that $T(2, d, 5)$ is non-empty and dense in $S_0(2, d, 5)$. 

 According to Lemma \ref{Butlerconjthrm3-1} and Proposition \ref{Butlerconjthrm4-1}, a general element of any irreducible component $X\subseteq S_0(3, d, 5)$ belongs to $T(3, d, 5)$. This together with Lemma \ref{nonemptiness2} implies that $T(3, d, 5)$ is non-empty and dense in $S_0(3, d, 5)$. 
\end{proof}

\begin{remark}
Our result on the Butler's
conjecture extends \cite[Theorem D]{CH2024} to a large range of degrees. The approach is also different, we use the linear equivalance of the coherent systems of type $(2,d,5)$ in that range and analyze the components of $S_0(2,d,5)$ and $S_0(3,d,5)$ from linear stability point of view, whereas in \cite{CH2024} the authors establish that, in certain range of the degree, the coherent systems of type $(2,d,5)$ is $\alpha$-stable for a large $\alpha$ in order to prove that the kernel bundle is semistable.
\end{remark}

\textit{Acknowledgements.} 
During this project, the first-named author was supported by a George Forster Fellowship at Humboldt-Universität zu Berlin. He is deeply grateful to Gavril Farkas and the Alexander von Humboldt Foundation for their generous support. 
He also visited George Harry Hitching at OsloMet University, and wishes to thank him for his warm hospitality and for the highly enjoyable and insightful conversations on this topic, as well as Abel Castorena for helpful discussions.

\end{document}